\documentclass{article}

\usepackage{xr}
\makeatletter
\usepackage{authblk}

\newcommand*{\addFileDependency}[1]{
\typeout{(#1)}
\@addtofilelist{#1}
\IfFileExists{#1}{}{\typeout{No file #1.}}
}\makeatother

\usepackage{comment}
 \usepackage{amsmath, amsthm, amssymb, bbm, setspace,bigints}
 \usepackage{mathrsfs}
  \usepackage{bm}
 \usepackage[margin=1 in]{geometry}
\usepackage{caption}
\usepackage{subcaption}
\usepackage[toc,page]{appendix}     
\usepackage{pdfpages}
\usepackage{epstopdf}
\usepackage{booktabs}
\usepackage{tabularx, multirow}
\usepackage{tikz, pgfplots}
\usetikzlibrary{positioning}
\usepackage{algorithm}
\usepackage{algpseudocode}
\usepackage{adjustbox}
\usepackage{booktabs}
\usepackage{siunitx}
\usepackage{comment}

\usepackage{empheq}
\usepackage{accents}
\usepackage{bm}
\newcommand*{\B}[1]{\ifmmode\bm{#1}\else\textbf{#1}\fi}

\newcommand{\be} {\begin{eqnarray*}}
\newcommand{\ee} {\end{eqnarray*}}

\newcommand{\argmin}{\mathop{\rm argmin~}}

\newcommand{\dd}{{\rm d}}

\newcommand{\wht}{\widehat}

\newcommand{\KL}{D_{\mbox{\scriptsize \rm KL}} }

\newcommand{\ri}{\textrm{(i)}}
\newcommand{\rii}{\textrm{(ii)}}
\newcommand{\riii}{\textrm{(iii)}}

\DeclareMathOperator{\supp}{supp}
\DeclareMathOperator{\osc}{Osc}
\DeclareMathOperator{\diag}{diag}

\newcommand{\algrule}[1][.2pt]{\par\vskip.5\baselineskip\hrule height #1\par\vskip.5\baselineskip}
\DeclareMathOperator{\var}{Var}

\def\m{\mathcal}
\def\ms{\mathscr}
\def\mb{\mathbb}

\def\mx{\mbox}
\def\tr{{\rm tr\,}}

\newcommand{\matnorm}[1]{{\left\vert\kern-0.25ex\left\vert\kern-0.25ex\left\vert #1 
    \right\vert\kern-0.25ex\right\vert\kern-0.25ex\right\vert}}
\newcommand{\norm}[1]{{\left\vert\kern-0.25ex\left\vert #1 
    \right\vert\kern-0.25ex\right\vert}}

\newtheorem{theorem}{Theorem}
\newtheorem{lemma}[theorem]{Lemma}

\newtheorem{rem}{Remark}

\newtheorem{ass}{Assumption}

\newcommand{\ry}[1]{\textcolor{blue}{#1}}

\title{Minimizing Convex Functionals over Space of Probability Measures via KL Divergence Gradient Flow}

\author{Rentian Yao}
\author{Linjun Huang}
\author{Yun Yang}
\affil{Department of Statistics, University of Illinois at Urbana-Champaign \authorcr Email: \{rentian2, linjunh2, yy84\}@illinois.edu}
\date{\vspace{-2em}}

\begin{document}
\maketitle

\begin{abstract}
Motivated by the computation of the non-parametric maximum likelihood estimator (NPMLE) and the Bayesian posterior in statistics, this paper explores the problem of convex optimization over the space of all probability distributions. We introduce an implicit scheme, called the implicit KL proximal descent (IKLPD) algorithm, for discretizing a continuous-time gradient flow relative to the Kullback-Leibler divergence for minimizing a convex target functional. We show that IKLPD converges to a global optimum at a polynomial rate from any initialization; moreover, if the objective functional is strongly convex relative to the KL divergence, for example, when the target functional itself is a KL divergence as in the context of Bayesian posterior computation, IKLPD exhibits globally exponential convergence. Computationally, we propose a numerical method based on normalizing flow to realize IKLPD. Conversely, our numerical method can also be viewed as a new approach that sequentially trains a normalizing flow for minimizing a convex functional with a strong theoretical guarantee.

\end{abstract}

\tableofcontents

\section{Introduction}\label{sec: intro}
Many problems in statistics and machine learning can be formulated as minimizing a functional, denoted as $\m F$, over the space $\ms P(\Theta)$ of all probability distributions on a (parameter) space $\Theta\subset \mb R^d$.
Examples include approximate Bayesian computation~\cite{dai2016provable, yao2022mean}, non-parametric estimation~\cite{yan2023learning}, deep learning~\cite{nitanda2022convex, chizat2022mean, chizat2018global}, and single-cell analysis in mathematical biology~\cite{lavenant2021towards}. Many recent studies consider addressing this optimization problem by numerically realizing the so-called Wasserstein gradient flow (WGF) for minimizing $\m F$, a continuous dynamics on $\ms P(\Theta)$ that evolves in the steepest descent direction of $\m F$ in the Wasserstein metric. WGFs inherit many appealing geometric interpretations from the conventional gradient flows in Euclidean space $\mathbb{R}^d$ and extend them to $\ms P(\Theta)$. However, a rapid convergence guarantee for WGF usually requires the displacement convexity of the objective functional $\m F$ (i.e., convexity of $\m F$ along Wasserstein geodesics, see Appendix~\ref{app:more_def} for a precise definition), which may impose more stringent conditions than the usual $L_2$ convexity of $\m F$ and therefore may not hold in many applications (such as Examples 1 and 2 below). Several works instead consider relaxing the displacement convexity condition to a PL-type inequality on $\m F$~\cite{chewi2020gradient, bolley2012convergence, cattiaux2010note, bolley2013uniform}. However, a PL-type inequality may not hold~\cite{yan2023learning} unless some impractical assumptions are imposed~\cite{nitanda2022convex}, or can be difficult to verify even for simple problems in the Euclidean setting \cite{wensing2020beyond} as it requires prior knowledge on the global optimum.

In this paper, we instead explore the use of Kullback-Leibler (KL) divergence gradient flow (KLGF) to minimize an $L_2$ convex target functional $\mathcal{F}$ over the space of all probability distributions $\ms P(\Theta)$.
In particular, we will focus on the following two motivating examples, where the target functionals are $L_2$ convex but not necessarily displacement convex.

\noindent\textbf{Example 1 (Non-parametric maximum likelihood estimation):}
The computation of the non-parametric maximum likelihood estimator (NPMLE) naturally arises in estimating the mixing distributions of mixture models and in using empirical Bayes methods to address compound decision problems. Concretely, we assume that the conditional distribution of a random variable $X$ given a parameter $\theta$ is $p(\,\cdot\,|\,\theta)$, where $\theta\in\Theta$ is drawn from an unknown mixing distribution $P^\ast$ in $\ms P(\Theta)$. Given $n$ i.i.d.~copies $X^n = (X_1, \cdots, X_n)$ of $X$, the NPMLE of $P^\ast$ is defined as
\begin{equation}\label{eqn: NPMLE}
\begin{aligned}
\wht P_n &= \argmin_{\rho\in\ms P(\Theta)}\m L_n(\rho),\quad\mx{with}\quad
\m L_n(\rho)\coloneqq \frac{1}{n}\sum_{i=1}^n -\log \Big(\int_\Theta p(X_i\,|\,\theta)\,\dd\rho(\theta)\Big),
\end{aligned}
\end{equation}
which minimizes the (averaged) negative log-likelihood functional $\m L_n:\ms P(\Theta)\to\mb R$. $\m L_n$ is obviously $L_2$ convex on $\ms P(\Theta)$ but not generally displacement convex (see Appendix~\ref{app:more_def} for an example).

The concept of NPMLE was first proposed in~\cite{kiefer1956consistency}, where they treated and estimated the mixing distribution as an infinite-dimensional object. When the parameter space $\Theta$ is the one-dimensional real line,~\cite{koenker2014convex} shows that $\wht P_n$ is a discrete probability measure with no more than $n$ atoms. Moreover, they propose a numerical method for solving NPMLE in $\mathbb{R}$ by utilizing a space discretization scheme to reformulate~\eqref{eqn: NPMLE} into a convex optimization problem, which can be efficiently solved using modern interior point methods. Consequently, their method is subject to the curse of dimensionality and becomes computationally demanding when dealing with multivariate parameters (see our numerical comparison in Section~\ref{sec: numerical}). When the true mixing distribution $P^\ast$ is a sub-Gaussian distribution on $\mb R$ and $p(\cdot\,|\,\theta)=\m N(\theta, 1)$, \cite{polyanskiy2020self} shows that the number of atoms in $\wht P_n$ reduces to $O(\log n)$. \cite{soloff2021multivariate} extends the optimality analysis of NPMLE to the multivariate and heteroscedastic normal observation model (with known heteroscedasticity), showing that, despite the possible non-uniqueness when $d\geq 2$, a solution with at most $n$ atoms exists.
For the multivariate Gaussian location mixture model where $p(\cdot\,|\,\theta) = \m N(\theta, I_d)$, \cite{yan2023learning} proposes an algorithm to solve~\eqref{eqn: NPMLE} based on discretizing a Wasserstein-Fisher-Rao (WFR) gradient flow~\cite{chizat2018interpolating,gallouet2017jko}, which can be numerically implemented using particle approximation.

\noindent\textbf{Example 2 (Bayesian posterior sampling):}
In Bayesian statistics, a core problem is sampling from the posterior distribution to estimate unknown parameters via the posterior mean and construct corresponding credible intervals, especially when exact computation of the posterior distribution is infeasible due to non-conjugacy. Given the prior density function $\pi(\theta)$ of the parameter $\theta\in\Theta$ and $n$ i.i.d.~samples $X^n = (X_1, \cdots, X_n)$ drawn from the likelihood function $p(\,\cdot\,|\,\theta)$, the posterior (density) is $\pi_n(\theta) \propto \pi(\theta)\prod_{i=1}^n p(X_i\,|\,\theta)$, which admits a variational characterization as
\begin{align}\label{eqn: posteior_sampling}
\begin{aligned}
&\pi_n = \argmin_{\rho\in\ms P(\Theta)} \int V_n(\theta)\,\dd\rho(\theta) + \int\rho\log\rho,
\quad\mx{where}\quad V_n(\theta) = -\log\pi(\theta) - \sum_{i=1}^n\log p(X_i\,|\,\theta)
\end{aligned}
\end{align}
denotes the effective potential function. In other words, the posterior can be identified as the global minimizer of the KL divergence functional $\KL(\,\cdot\,\|\,\pi_n)$ up to an additive constant. The KL functional is always $L_2$ convex on $\ms P(\Theta)$; but the displacement convexity requires more conditions such as the convexity of $V_n$.

Beyond the classical MCMC algorithms~\cite{tierney1994markov}, some recent advancements in sampling from Bayesian posterior distributions rely on discretizing certain gradient flows in the space of all probability distributions.  One approach is based on the WGF by discretizing its induced stochastic differential equation, namely, the Langevin dynamics. For example,~\cite{dalalyan2017theoretical} proposes the (unadjusted) Langevin Monte Carlo algorithm which discretizes the Langevin dynamics via an explicit scheme. However, this algorithm is known to produce a non-vanishing (asymptotically) bias \cite{wibisono2018sampling, dalalyan2017further} due to the explicit discretization, and is improved to be unbiased via a forward-backward discretization scheme later in~\cite{wibisono2018sampling}. However, the fast convergence of these iterative algorithms based on discretizing the WGF requires imposing stringent conditions on $\pi_n$, such as log-concavity, 
isoperimetry, or log-Sobolev inequalities~\cite{dalalyan2017theoretical,wibisono2018sampling,chewi2021analysis}. On another track, \cite{dai2016provable} proposes a stochastic particle mirror descent algorithm to iteratively approximate the Bayesian posterior density.

\noindent {\bf Our contributions.} In this work, we propose an implicit scheme, called the implicit KL proximal descent (IKLPD) algorithm, for discretizing a continuous-time gradient
flow relative to the Kullback-Leibler divergence for minimizing a general $L_2$ convex
functional $\m F$.  We show that, under the $L_2$ convexity condition alone, IKLPD converges to a global optimum at a polynomial rate from any initialization that admits a density; moreover, if $\m F$ is strongly convex relative to the KL divergence, for example, when $\m F$ itself is a KL divergence as in the context of Bayesian posterior computation, IKLPD exhibits globally exponential convergence. Therefore, the proposed implicit scheme avoids imposing any smoothness condition on the $L_2$-gradient of $\m F$, as is typically required by an explicit discretization scheme; and a low smoothness level adds a strong constraint on the learning rate (or step size) of the algorithm. Moreover, it is noteworthy that, unlike functions over Euclidean space where a Lipschitz gradient condition is generally not overly stringent, a Lipschitz $L_2$-gradient condition on $\m F$ can either rule out many commonly used functionals, such as the KL divergence, or require substantial effort to verify. Our development can also be extended to a general (implicit) proximal mirror descent algorithm with a Bregman divergence beyond the KL; see Appendix~\ref{app:general_divergence} for more details.

Computationally, we propose a numerical method based on normalizing flow \cite{dinh2016density,rezende2015variational,kingma2018glow,papamakarios2021normalizing} to implement IKLPD. The compositional structure of the normalizing flow aligns perfectly with the iterative nature of our time-discretization algorithm. Specifically, we sequentially stack the local short normalizing flow, learned within each IKLPD iteration, to form a global, layered normalizing flow for approximating a minimizer of $\m F$. Alternatively, our algorithm can also be viewed as a new approach that sequentially trains a normalizing flow for minimizing a convex functional $\m F$ over $\ms P(\Theta)$ with a strong theoretical guarantee. When employed for computing the NPMLE and Bayesian posteriors, our method exhibits promising performance compared to explicit schemes and other specialized competing algorithms.

We also consider two extensions of our development. In the first extension, we allow nonzero numerical error to occur when solving each implicit step and investigate how these errors accumulate (Theorem~\ref{thm: error tol}), which provides guidance on the design of stopping criteria in solving the implicit step.
In the second extension, we propose and analyze the convergence of a stochastic version of IKLPD (Theorem~\ref{thm: stochastic}), which is useful in practical applications where the sample size is large. To our knowledge, this is the first study that analyzes a stochastic proximal type algorithm for optimizing functionals on the space of all probability distributions.


\noindent {\bf More related works.}
\cite{ying2020mirror} applies a mirror descent algorithm for minimizing an interacting free energy over $\ms P(\Theta)$ composed of a potential energy, a KL divergence and a self-interaction energy; however, they do not provide any convergence analysis.
\cite{aubin2022mirror,chizat2021convergence} prove the explicit convergence rate of the mirror descent for minimizing general (strongly) convex functionals over the space of all probability distributions.
\cite{chizat2021convergence} studies the convergence of the mirror descent algorithm for minimizing a special class of composite convex targets $\m F$ whose primary component depends on $\rho\in\ms P(\Theta)$ through a linear functional. When specializing the Bregman divergence to the KL, their algorithm can be viewed as an explicit scheme to discretize the KL gradient flow. As a result, their theory requires $\m F$ to have a Lipschitz $L_2$-gradient and does not cover common $f$-divergences \cite{renyi1961measures} such as the KL. \cite{aubin2022mirror} proposes a different smoothness characterization called relative smoothness, which is analogous to the Euclidean case smoothness characterization via quadratic bounds. However, they only verify their conditions for the KL functional, with applications to the entropic optimal transport and Expectation Maximization (EM). In addition, their convergence bound diverges to infinity as the (global) minimizer of $\m F$ becomes singular (i.e., does not admit a density). Moreover, these two papers \cite{chizat2021convergence,aubin2022mirror} do not provide concrete numerical methods to implement their algorithms.


As we mentioned earlier,~\cite{koenker2014convex, soloff2021multivariate, yan2023learning}
propose some state-of-the-art algorithms for numerically computing the NPMLE. However, the WFR based method by \cite{yan2023learning} only applies to the Gaussian location mixture model and does not have an explicit convergence rate guarantee; the convex optimization based algorithms by \cite{koenker2014convex, soloff2021multivariate} approximate the target distribution through histograms by space discretization, and therefore suffers from the curse of dimensionality. In comparison, our method has the worse case $O(k^{-1})$ convergence guarantee after $k$ iterations, and tends to be scalable to higher dimensions.
For Bayesian posterior computation, MCMC is known to exhibit slow mixing in complex or high-dimensional problems, and most existing numerical algorithms based on Langevin dynamics require stringent conditions such as log-concavity, isoperimetry, or log-Sobolev inequalities~\cite{dalalyan2017theoretical,wibisono2018sampling,chewi2021analysis} to guarantee fast convergence. In comparison, our algorithm guarantees exponential convergence without imposing any conditions on the target posterior, as long as the implicit step can be efficiently implemented, which is true at least in our concerned examples.
Additional literature review on mirror descent and stochastic (proximal) mirror descent in the Euclidean space, along with optimization algorithms on the space of all probability distributions, can be found in the supplementary material.


\section{KL Divergence Gradient Flow and Implicit Time Discretization} \label{sec: IKLPD}

To begin with, we briefly introduce some useful definitions. Let $\m F:\ms P(\Theta)\to \mb R$ be a lower semi-continuous functional and $\ms P^r(\Theta)$ denote the set of all probability distributions admitting a density on $\Theta$.
Under mild conditions (see Appendix~\ref{app:more_def} for details), one can define the \emph{first variation} of $\m F$ at $\rho\in \ms P(\Theta)$ as a map $\frac{\delta\m F}{\delta\rho}(\rho): \Theta\to\mb R$ such that for any perturbation $\chi=\rho' - \rho$ with $\rho'\in\ms P^r(\Theta)$,
\begin{align*}
\frac{\dd}{\dd\varepsilon}\m F(\rho+\varepsilon\chi)\bigg|_{\varepsilon=0} = \int_{\Theta}\frac{\delta\m F}{\delta\rho}(\rho)\,\dd\chi.
\end{align*}
Note that $\frac{\delta\m F}{\delta\rho}(\rho)$ is only uniquely defined up to an additive constant. The first variation can be viewed as the $L_2$-gradient of $\m F$ in $\ms P(\Theta)$. A functional $\m F$ is called
\emph{$\lambda$-relative strongly convex} (relative to KL) if for any pair of regular probability measures $\rho,\,\rho'\in\ms P(\Theta)$ ($\ms P^r(\Theta)$ when $\lambda > 0$) such that $\m F(\rho)$ is finite, 
\begin{align*}
\m F(\rho') \geq \m F(\rho) + \int\frac{\delta\m F}{\delta\rho}(\rho)\,\dd(\rho'-\rho) + \lambda\KL(\rho'\,\|\,\rho).
\end{align*}
We simply say $\m F$ to be ($L_2$-)convex if $\m F$ satisfies the above inequality with $\lambda =0$. Note that the NPMLE example in Section~\ref{sec: intro} has a convex $\m F$, and the Bayesian posterior example therein has a $1$-relative strongly convex $\m F$; see Appendix~\ref{app:convexity} for a proof.

\begin{rem} 
The $L_2$ convexity and the displacement convexity are not directly comparable. For example, the KL divergence functional in~\eqref{eqn: posteior_sampling} is always $L_2$ convex but not displacement convex unless potential $V_n$ is a convex function over $\Theta$. Conversely, the self-interaction energy functional $\m W(\rho) = \int_{\mb R^2}(x-y)^2\,\dd\rho(x)\dd\rho(y)$ is displacement convex due to the convexity of the square function~\cite{mccann1997convexity}; however, direct calculation yields
$\frac{1}{2}\big(\m W(\rho) + \m W(\rho')\big) - \m W\big(\frac{1}{2}(\rho + \rho')\big)= \frac{1}{2}\big(\int_{\mb R} x\,\dd(\rho-\rho')(x)\big)^2\leq 0$, indicating that $\m W$ is instead $L_2$-concave.
\end{rem}

Given an initialization $\rho_0\in\ms P^r(\Theta)$, we consider the following iterative scheme for minimizing $\m F$ with step size $\{\tau_k:\,k\geq 1\}$, 
\begin{align}\label{eqn: IKLPD}
\rho_{k} = \argmin_{\rho\in\ms P(\Theta)} \m F(\rho) + \frac{1}{\tau_k}\KL(\rho\,\|\,\rho_{k-1}), \,\, k\geq 1,
\end{align} 
which will be referred to as the implicit KL proximal descent (IKLPD) algorithm. In Section~\ref{sec: NF}, we propose using a normalizing flow~\cite{kobyzev2020normalizing} to numerically optimize the objective in the implicit step~\eqref{eqn: IKLPD}. Note that this implicit step optimization problem becomes easier as the step size $\tau_k$ becomes smaller, as the optimal solution $\rho_k$ is expected to become closer to the previous iterate $\rho_{k-1}$ (e.g.,~$\KL(\rho_k\,\|\,\rho_{k-1})=O(\tau_k)$), so that a few (stochastic) gradient iterations are sufficient to produce a relatively good solution. In contrast, as $\tau_k\to \infty$, implementing the implicit step becomes as hard as solving the original problem of minimizing $\m F$. We conduct a numerical experiment in Section~\ref{sec: numerical} to explore the impact of the step size $\tau_k$ on the implicit step computation and the overall convergence of the IKLPD algorithm.

It is worth noting that IKLPD extends the implicit gradient descent method for minimizing a function $f$ on $\Theta$ under Euclidean $\ell_2$ metric $\|\cdot\|$, 
\begin{align*}
    x_k = \argmin_{x\in \Theta} f(x) + \frac{1}{2\tau_k}\|x-x_{k-1}\|^2, \,\, k\geq 1, 
\end{align*}
which is also the proximal point method \cite{rockafellar1997convex,boyd2004convex} with convex function $\frac{1}{2}\|\cdot\|^2$; in particular, IKLPD changes the discrepancy measure $\frac{1}{2}\|x-x_{k-1}\|^2$ with $\KL(\rho\,\|\,\rho_{k-1})$. More generally, we may also consider a broader class of implicit mirror descent algorithms by substituting the KL with a general Bregman divergence, such as $L_2$ distance, Itakura--Saito divergence~\cite{savchenko2019itakura}, and hyperbolic divergence~\cite{ghai2020exponentiated}. The key property of Bregman divergences used in the proof is the ``three-points identity" (e.g.,~Lemma 3.1 in~\cite{chen1993convergence}), which connects the first variation of the objective~\eqref{eqn: IKLPD} with the Bregman divergence. 
However, our considered KL is often better aligned with the information geometry inherent to statistical problems. In contrast, other common divergences in statistics, such as the $\chi^2$ divergence and the R\'{e}nyi divergence are not Bregman divergences (see Appendix~\ref{app:general_divergence}).

Analogous to gradient (or mirror) descent in Euclidean space~\cite{krichene2015accelerated}, which can be interpreted as discretizing a continuous-time gradient flow on $\Theta$, IKLPD also corresponds to employing an implicit discretization scheme for the KL gradient flow (KLGF) on $\ms P(\Theta)$, which is described by the ordinary differential equation (ODE)
\begin{align}\label{eqn: IKLPD_cts}
\frac{\dd}{\dd t}\big(\log\rho_t\big)= - \frac{\delta\m F}{\delta\rho}(\rho_t) + \int_\Theta\frac{\delta\m F}{\delta\rho}(\rho_t)(\theta)\,\dd\rho_t(\theta).
\end{align}
This dynamic is also known as the Fisher-Rao gradient flow~\cite{bauer2016uniqueness, yan2023learning}. Let $\rho^\ast$ be a global minimum of $\m F$.
The following theorem shows the convergence of KLGF for an $L_2$ convex functional $\m F$.

\begin{theorem}\label{thm: FR_dynamics}
Assume $\m F$ to be a $\lambda$-relative strongly convex functional with respect to the KL divergence and $\rho^\ast\in\ms P^r(\Theta)$. If $\lambda >0$, then $\rho_t$ satisfies
\begin{align*}
\KL(\rho^\ast\,\|\,\rho_t) \leq e^{-\frac{\lambda t}{2}}\KL(\rho^\ast\,\|\,\rho_0);
\end{align*}
if $\lambda = 0$, then $\bar\rho_t = \frac{1} {t}\int_0^t\rho_s\,\dd s$ satisfies
\begin{align*}
\m F(\bar{\rho}_t) - \m F(\rho^\ast) \leq \frac{1}{t}\KL(\rho^\ast\,\|\,\rho_0).
\end{align*}
\end{theorem}

\section{Theoretical Results}
\label{sec: main_results}
We analyze the convergence of IKLPD and two variants: an inexact IKLPD that permits non-zero numerical errors when solving the implicit step~\eqref{eqn: IKLPD}, and a stochastic version of IKLPD.

\subsection{Convergence of IKLPD}
We make the following assumptions.
\begin{ass}[Existence of IKLPD iterates]\label{assump: exist}
For each $k\geq 1$, the solution $\rho_k$ as the $k$-th iteration of IKLPD algorithm as defined by~\eqref{eqn: IKLPD} exists.
\end{ass}
Assumption~\ref{assump: exist} is typically verifiable by applying Prokhorov's Theorem~\cite{prokhorov1956convergence} when $\mathcal{F}$ is continuous with respect to the weak topology of $\ms {P}(\Theta)$. The continuity of $\mathcal{F}$ holds for many models, including the NPMLE discussed in Section~\ref{sec: intro}.

\begin{ass}[Relative strong convexity]\label{assump: convexity}
$\m F$ is $\lambda$-relative strongly convex on $\ms P(\Theta)$ for $\lambda \geq 0$.
\end{ass}
Assuming some convexity condition is standard and necessary in the convergence analysis of proximal type algorithms~\cite{dragomir2021fast, aubin2022mirror}.

\begin{theorem}\label{thm: discrete_regular}
Suppose Assumptions~\ref{assump: exist} and~\ref{assump: convexity} hold and $\rho^\ast\in\ms P^r(\Theta)$.
\newline
\noindent (1) If $\lambda > 0$ and $\tau_k \equiv \tau > 0$, then we have
\begin{align*}
\KL(\rho^\ast\,\|\,\rho_k)\leq \Big(1 + \frac{\lambda\tau}{2}\Big)^{-k}\KL(\rho^\ast\,\|\,\rho_0).
\end{align*}
\newline
\noindent (2) If $\lambda = 0$, then we have
\begin{align*}
\min_{1\leq \ell\leq k}\m F(\rho_\ell) - \m F(\rho^\ast) \leq \frac{1}{\sum_{\ell=1}^k \tau_\ell}\KL(\rho^\ast\,\|\,\rho_0).
\end{align*}
\end{theorem}
\begin{rem}
Several remarks are in order. First, when $\rho^\ast\in\ms P^r(\Theta)$, we do not need any extra condition beyong the convexity to guarantee the convergence of IKLPD. As we discussed in the introduction, this is different from the explicit discretization scheme considered in~\cite{chizat2021convergence,aubin2022mirror}, which require additional smoothness conditions. Second, our proof for the $\lambda =0$ case also implies the same convergence bound to hold for the weighted trajectory average $\bar\rho_k=\big(\sum_{\ell=1}^k \tau_\ell\big)^{-1}\sum_{\ell=1}^k \tau_\ell \,\rho_\ell$. Third, if $\lambda=0$ and $\tau_k=\tau$, then the IKLPD exhibits an $O(k^{-1})$ convergence rate after $k$ iterations, which matches the convergence rate of the Euclidean proximal mirror descent algorithm for minimizing a smooth and convex function (e.g. Theorem 10.81 in~\cite{beck2017first}).
\end{rem}

Theorem~\ref{thm: discrete_regular} requires $\rho^\ast$ to admit a density, so that the initial KL divergence $\KL(\rho^\ast\,\|\,\rho_0)$ is finite. However, in many applications, such as the NPMLE computation, $\rho^\ast$ can contain singular components or can even be a discrete measure~\cite{polyanskiy2020self}. In these cases, 
Assumption~\ref{assump: convexity} can only hold with  $\lambda=0$. 
To see this, we can apply the $\lambda$-relative strong convexity to $\rho=\rho_0$ for any $\rho_0\in\ms P^r(\Theta)$ with a bounded $\frac{\delta\m F}{\delta\rho}(\rho_0)$, and $\rho'=\rho^\sigma = \rho^\ast\ast\m N(0, \sigma^2I_d)\in\ms P^r(\Theta)$, the convolution of $\rho^\ast$ with a normal distribution. This yields  $\lambda\KL(\rho^\sigma\,\|\,\rho_0) \leq \m F(\rho^\sigma) - \m F(\rho_0) - \int\frac{\delta\m F}{\delta\rho_0}(\rho_0)\,\dd(\rho^\sigma-\rho_0) $. As we let $\sigma\to0^+$, the right-hand side of this inequality is finite, while the KL term $\KL(\rho^\sigma\,\|\,\rho_0)$ diverges when $\rho^\ast\not\in \ms P^r(\Theta)$, indicating that $\lambda = 0$. 
To extend the convergence result to such $\rho^\ast$ that does not admit a density, we need an additional assumption about the continuity of $\m F$ around $\rho^\ast$. Let $W_1$ denote the $1$-Wasserstein metric; see Appendix~\ref{app:more_def} for a precise definition.
\begin{ass}[Local $W_1$-continuity]\label{assump: smooth_first_variation}
There exists a constant $L > 0$ such that
\begin{align*}
|\m F(\rho) - \m F(\rho^\ast)| \leq L W_1(\rho,\, \rho^\ast),\quad\forall\,\rho\in\ms P(\Theta).
\end{align*}
\end{ass}
This local continuity condition on $\mathcal{F}$ is less stringent than a typical smoothness condition assumed in the analysis of explicit schemes that involves the first variation, and it is satisfied in our examples.


\begin{theorem}\label{thm: discrete}
If Assumptions~\ref{assump: exist}, \ref{assump: convexity} and~\ref{assump: smooth_first_variation} hold with $\lambda=0$ and $\rho^\ast\in\ms P(\Theta)$, then for any $\rho\in \ms P^r(\Theta)$, 
\begin{align*}
\min_{1\leq \ell\leq k}\m F(\rho_\ell) - \m F(\rho^\ast)
\leq\frac{\KL(\rho\,\|\,\rho_0)}{\sum_{\ell=1}^k\tau_\ell} + \frac{L}{2} W_1(\rho,\,\rho^\ast).
\end{align*}
\end{theorem}
\begin{rem}
Theorem~\ref{thm: discrete} suggests that when $\rho^\ast$ contains singular components, the convergence rate of IKLPD may depend on finer structures on the singularity of $\rho^\ast$ as we want to construct some $\rho$ to compensate for the singularity. For example, if $\rho^\ast$ is a discrete measure, then the convergence rate is $O(\frac{d\log k}{k})$; generally, if $\rho^\ast$ is supported on a $d'$-dimensional hyperplane in the ambient space $\Theta\subset \mb R^d$ with $d'<d$, then the convergence rate becomes $O(\frac{(d-d')\log k}{k})$ (the support of a discrete measure has an effective dimension $d'=0$); see Appendix~\ref{app: discrete_proof} for a proof, where we choose $\rho$ in the theorem as the convolution of $\rho^\ast$ and a $(d-d')$-dim standard Gaussian distribution whose variance is optimized to make the upper bound smallest, so that the convoluted distribution admits a density.
\end{rem}

\subsection{Convergence of Inexact IKLPD}
For inexact IKLPD, we allow non-zero numerical errors when solving the implicit step~\eqref{eqn: IKLPD}, and study their impact on overall convergence and the design of the implicit step stopping criterion. In practice, one can use the first-order optimility condition $\frac{\delta\m F}{\delta\rho}(\rho) + \frac{1}{\tau_k}\log\frac{\rho}{\rho_{k-1}}=$constant to design stopping criterion and monitor the convergence of the implicit step optimization sub-problem~\eqref{eqn: IKLPD}. Specifically, let $\{\rho_k^{\rm err}:\,k\geq 0\}$ denote the iterates from an inexact IKLPD, and let
\begin{align}\label{eqn: first variation iterate}
\eta_k(\cdot) \coloneqq \frac{\delta\m F}{\delta\rho}(\rho_k^{\rm err})(\cdot) + \frac{1}{\tau_k}\log\frac{\rho_k^{\rm err}}{\rho_{k-1}^{\rm err}}(\cdot)
\end{align}
denote the first variation (as a function over $\Theta$) of the target functional in the implicit step~\eqref{eqn: IKLPD} evaluated at $\rho_k^{\rm err}$.
Let $\{\varepsilon_k:\,k\geq 1\}$ denote a generic sequence of error tolerance levels. For technical convenience, we characterize the convergence of each implicit step optimization via the oscillation of $\eta_k$, and make the following assumption.

\begin{ass}[Uniform error control]\label{assump: error}
For each $k\geq 1$ and $\varepsilon_k \geq 0$, we have $\osc_\Theta(\eta_k) \leq \varepsilon_k$, where
\begin{align*}
\osc_\Theta(\eta_k)\coloneqq  \sup_{\theta, \theta'\in\Theta}\|\eta_k(\theta)-\eta_k(\theta')\|
\end{align*}
is the oscillation of $\eta_k$ over $\Theta$.
\end{ass}
There are also other types of inexact algorithms for optimizing functionals on the space of all probability distributions \cite{dai2016provable,kent2021frank}, some of which are not implementable since they require knowledge of unknown quantities, such as the exact solution of the subproblem, to evaluate the tolerance metric. In our context, one can also use other characterizations, such as the variance of $\eta_k$ under $\rho_{k-1}^{\rm err}$ that is easier to compute in practice. 

The following theorem illustrates the impact of error tolerance level on the convergence rate of inexact IKLPD when $\m F$ is $\lambda$-relative strongly convex for $\lambda>0$.
In particular, we consider two regimes: $\varepsilon_k$ has either an exponential decay or a polynomial decay in $k$; and the inexact IKLPD exhibits different convergence patterns under the two regimes.
\begin{theorem}\label{thm: error tol}
Suppose Assumption~\ref{assump: convexity} holds with $\lambda > 0$ and Assumption~\ref{assump: error} also holds, and consider $\tau_k \equiv  \tau$.
\newline
\noindent (1) If $\varepsilon_k \leq \kappa\varepsilon^k$ for some $\kappa > 0$ and $0 < \varepsilon < 1$ satisfying $\varepsilon\sqrt{1+\lambda\tau/2}\neq 1$, then there exists a constant $C = C(\tau, \lambda, \varepsilon) > 0$ such that
\begin{align*}
\KL(\rho^\ast\,\|\,\rho_k^{\rm err}) \leq \frac{C\kappa^2 + 2\KL(\rho^\ast\,\|\,\rho_0)}{(\min\{\varepsilon^{-2}, 1+\lambda\tau/2\})^k};
\end{align*}
\newline
\noindent (2) If $\varepsilon_k \leq \varepsilon k^{-\alpha}$ for some $\varepsilon,\, \alpha > 0$, then there exists a constant $C = C(\tau, \lambda, \alpha) > 0$ such that
\begin{align*}
\KL(\rho^\ast\,\|\,\rho_k^{\rm err}) \leq \frac{2\KL(\rho^\ast\,\|\,\rho_0)}{(1+\lambda\tau/2)^{k}} + \frac{C\varepsilon^2}{k^{2\alpha}}
\end{align*}
\end{theorem}
\begin{rem}
Similar to Theorem~\ref{thm: discrete}, Theorem~\ref{thm: error tol} requires $\rho^\ast$ to have a density when $\lambda > 0$. Our theorem cannot cover the $\lambda=0$ case, since in order to show $\KL(\rho^\ast\,\|\,\rho_k^{\rm err})$ is decreasing in $k$, we need the relative strong convexity to contribute a term that compensates for the error caused by  $\eta_k$. In addition, the current proof of Theorem~\ref{thm: error tol} can only be extended to cover a Bregman divergence that dominates the $L_1$ distance, such as any divergences stronger than the KL, since we need to use it to address an additional error term that depends on the $L_1$ distance between $\rho_k^{\rm err}$ and $\rho^\ast$.
\end{rem}

\subsection{Convergence of Stochastic IKLPD}
In this section, we propose and analyze a stochastic version of IKLPD, whose $k$-th iterate is given by
\begin{align}\label{eqn: stoch IKLPD}
\rho_{k}^{\rm stoc} = \argmin_{\rho\in\ms P(\Theta)}\m F_{\xi_k}(\rho) + \frac{1}{\tau_k}\KL(\rho\,\|\,\rho_{k-1}^{\rm stoc}).
\end{align}
Here $\m F_{\xi_k}$ is an unbiased estimator of $\m F$ for any fixed input in $\ms P(\Theta)$, with $\xi_k$ indicating the source of randomness in iteration $k$. For example, in a statistical setting such as NPMLE, $\m F_{\xi_k}$ can be the negative log-likelihood functional over a random selected mini-batch. To prove the convergence, we make the following Assumption. 

\begin{ass}[Stochastic IKLPD]\label{assump: stoch IKLPD}
The stochastic objective functional $\m F_{\xi}$ satisfies:
\newline
\noindent (1) (Unbiasedness) $\mb E_\xi\big[\m F_\xi(\rho)\big] = \m F(\rho)$.
\newline
\noindent (2) (Solution existence) A solution of~\eqref{eqn: stoch IKLPD} exists.
\newline
\noindent (3) (Randomness condition) $\{\xi_k: k\geq 1\}$ are independently and identically distributed.
\newline
\noindent (4) (One-sided relative Lipschitz continuity) For some $L(\xi)$ with a finite second-order moment,
\begin{align*}
\m F_\xi(\rho) - \m F_\xi(\rho') \leq L(\xi)\sqrt{\KL(\rho'\,\|\, \rho)}
\end{align*}
holds for every $\rho, \rho'\in\ms P(\Theta)$.
\end{ass}
The one-sided relative Lipschitz continuity condition is also considered by~\cite{bertsekas2011incremental, davis2018stochastic}, which was utilized to analyze the convergence of stochastic proximal descent and stochastic proximal mirror descent in Euclidean space. In our proof, this condition is used to bound the difference of $\m F_{\xi_k}(\rho_{k}^{\rm stoc})$ and $\m F_{\xi_k}(\rho_{k-1}^{\rm stoc})$.

\begin{theorem}\label{thm: stochastic}
Assume that $\m F_\xi$ is $\lambda$-relative strongly convex for $\lambda\geq 0$. Suppose Assumption~\ref{assump: stoch IKLPD} holds and $\rho^\ast\in\ms P^r(\Theta)$. Let $\tau>0$ be a constant.
\newline
\noindent (1) If $\lambda=0$, then by taking $\tau_k = \frac{\tau}{\sqrt{k+1}}$ we have
\begin{align*}
&\min_{0\leq \ell\leq k-1} \mb E\big[\m F(\rho_\ell^{\rm stoc})\big] - \m F(\rho^\ast) \leq \frac{4\KL(\rho^\ast\,\|\,\rho_0) + \tau^2\log(k+1)\mb E [L(\xi_1)^2]}{8\tau(\sqrt{k+1}-1)};
\end{align*}
\newline
\noindent (2) If $\lambda>0$, then by taking $\tau_k = \frac{2}{\lambda(k+1)}$ we have
\begin{align*}
&\min_{0\leq \ell\leq k-1}\mb E\m F(\rho_\ell^{\rm stoc}) - \m F(\rho^\ast) \leq \frac{2\lambda^2\KL(\rho^\ast\,\|\,\rho_0) + \log(k+1)\mb E[L(\xi_1)^2]}{2\lambda k}.
\end{align*}
\end{theorem}
\begin{rem}
The convergence rates in our theorem match those of stochastic gradient descent~\cite{rakhlin2011making, nemirovski2009robust} and stochastic (proximal) mirror descent~\cite{davis2018stochastic, lan2020first} for minimizing (strongly) convex functions in Euclidean space. Additionally, when Assumption~\ref{assump: smooth_first_variation} holds, the same smoothing argument as in the proof of Theorem~\ref{thm: discrete} can be carried over to deal with a singular $\rho^\ast\not\in\ms P^r(\Theta)$.
\end{rem}

\section{Computation via Normalizing Flow}\label{sec: NF}
We propose using normalizing flow (NF) to solve the implicit step optimization problem~\eqref{eqn: IKLPD}.
Normalizing flows~\cite{dinh2016density,rezende2015variational,kingma2018glow,papamakarios2021normalizing} offer a general mechanism for defining expressive probability distributions thorough transforming a simple probability distribution into a complex one using compositions of invertible and differentiable transformations. For simplicity, we will refer to the IKLPD steps as the outer loop (iterations), and the (stochastic) gradient steps for optimizing the NF parameters in the implicit scheme problem~\eqref{eqn: IKLPD} as the inner loop (iterations). 

Given the shared compositional structure between our iterative IKLPD algorithm and the NF, we propose sequentially stacking the local, short normalizing flows, learned within each inner-loop iteration, to form a global, layered normalizing flow for approximating $\rho^\ast$. Concretely, we use $T_\#\rho$ to denote the pushforward distribution of a distribution $\rho\in\ms P(\Theta)$ through a transport map $T:\,\Theta\to\Theta$, and use $\wht T^{(k)}$ to denote the local normalizing flow learned through solving~\eqref{eqn: IKLPD}, yielding $\rho_k=\wht T^{(k)}_\#\rho_{k-1}$, where
\begin{align*}
\wht T^{(k)} = \argmin_{T\in \m T} \m F\big(T_\# \rho_{k-1}\big) + \frac{1}{\tau_k}\KL\big(T_\# \rho_{k-1}\,\big\|\,\rho_{k-1}\big), 
\end{align*} 
Here, $\m T$ denotes a generic normalizing flow class. 

Note that another benefit of using NF here is that the KL term can be directly computed in terms of a closed form expression of the log-density of $T_\# \rho_{k-1}$, whereas other numerical methods based on particle approximation require the use of kernel density estimation to approximate this density; further details of its numerical computation using (stochastic) gradient descent and the reparametrization trick are provided in Appendix~\ref{app:computation}. With these local NF maps, we can use the telescoping trick to express $\rho_k=\wht T^{(k)}_\#\circ \cdots\circ \wht T^{(1)}_\# \rho_0$, which defines a generative process for sampling from $\rho_k$. As $k$ increases, to maintain a fixed storage budget (e.g.,~keep at most $k_0$ local NFs), one may employ a teacher-student architecture~\cite{hu2022teacher,hinton2015distilling} to distill knowledge by utilizing a single NF to compress all historical local NFs beyond the most recent $(k_0-1)$ ones; see Appendix~\ref{app:computation} for a simple illustration.

\section{Numerical Results}\label{sec: numerical}

For the implementation, we used the Python \texttt{normflows} package \cite{Stimper2023} based on PyTorch to implement the real-valued non-volume preserving (real-NVP) normalizing flow \cite{dinh2016density} for our method.
We consider three examples: NPMLE for Gaussian location mixture model, NPMLE for Gaussian location scale mixture model, and sampling from a distribution known up to a constant (Bayesian computation). For NPMLE, we also consider two state-of-the-art competing methods, the Wasserstein-Fisher-Rao (\texttt{WFR}) gradient flow~\cite{yan2023learning} and a convex optimization based method \cite{koenker2014convex} (referred to as the \texttt{KW} method). For the Bayesian computation example, we compare our method with the (unadjusted) Langevin Monte Carlo algorithm (\texttt{Langevin}), which corresponds to an explicit discretization scheme to the Wasserstein gradient flow. Due to space constraints, we defer the details about the implementations and setup of each example below, as well as additional plots and results, to Appendix~\ref{app:computation}.

\begin{figure*}[t]
\centering
   \begin{subfigure}{0.32\linewidth} \centering
     \includegraphics[scale=0.27]{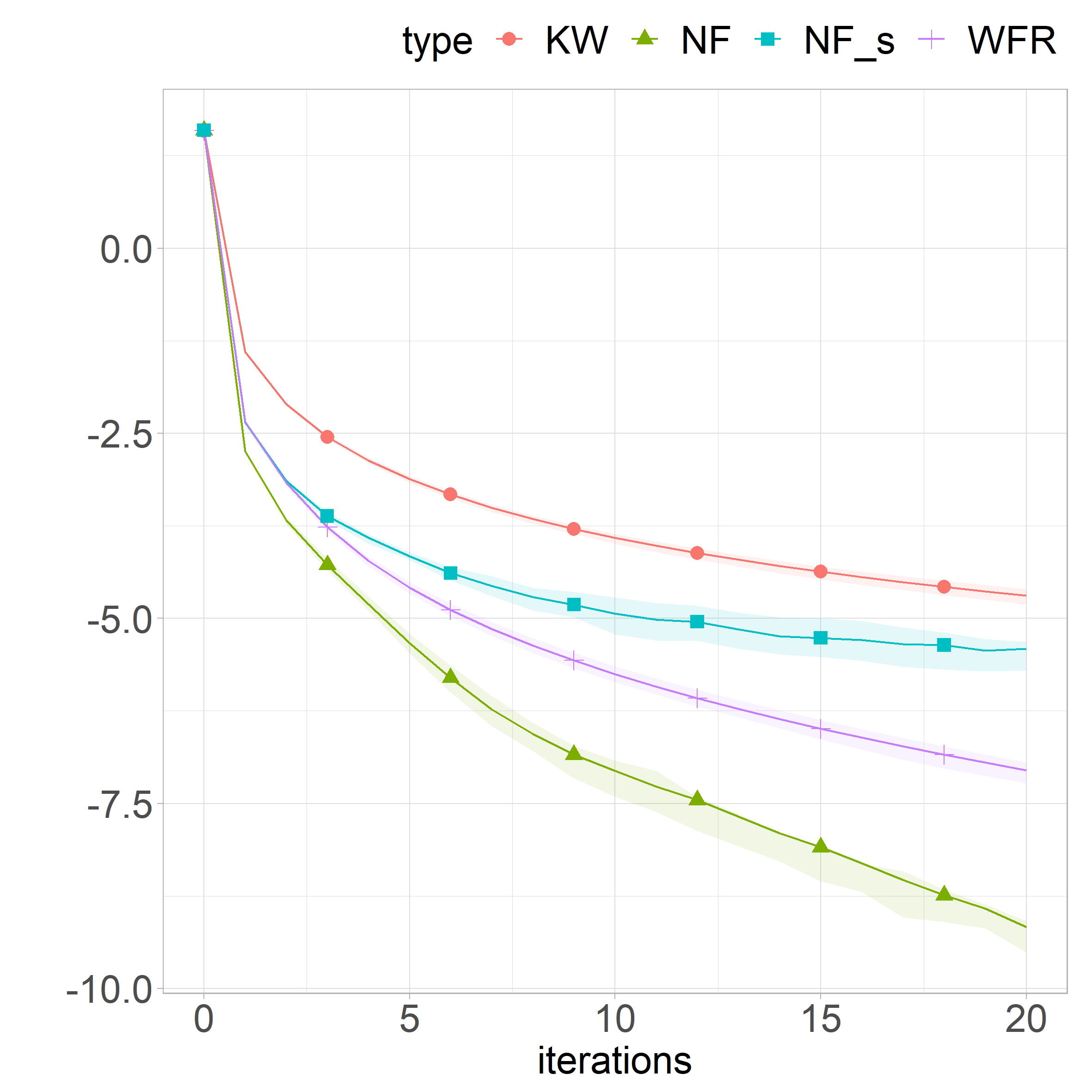}
     \caption{Location Mixture}
   \end{subfigure}
   \begin{subfigure}{0.32\linewidth} \centering
     \includegraphics[scale=0.27]{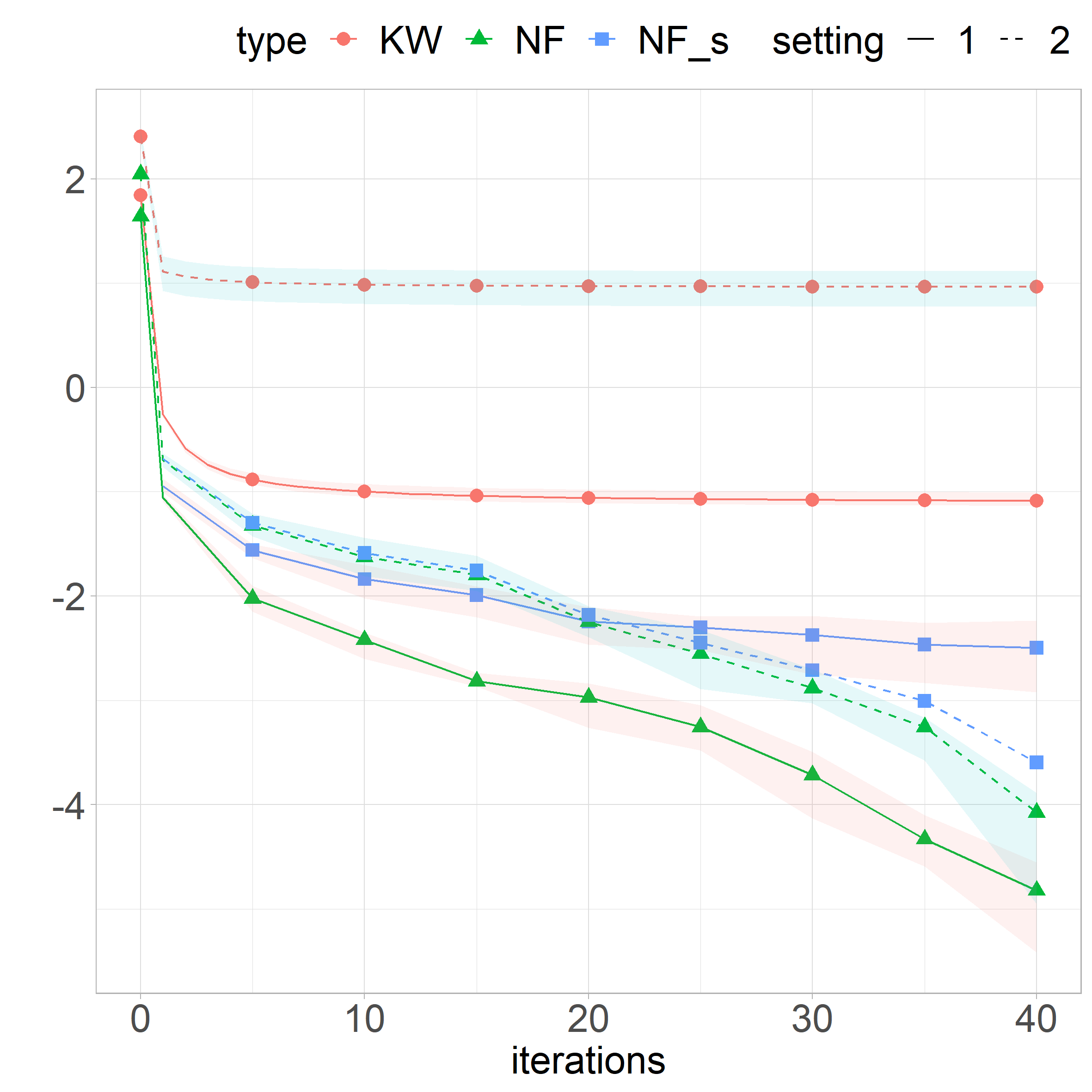}
       \caption{Location Scale Mixture}
   \end{subfigure}
   \begin{subfigure}{0.32\linewidth} \centering
     \includegraphics[scale=0.27]{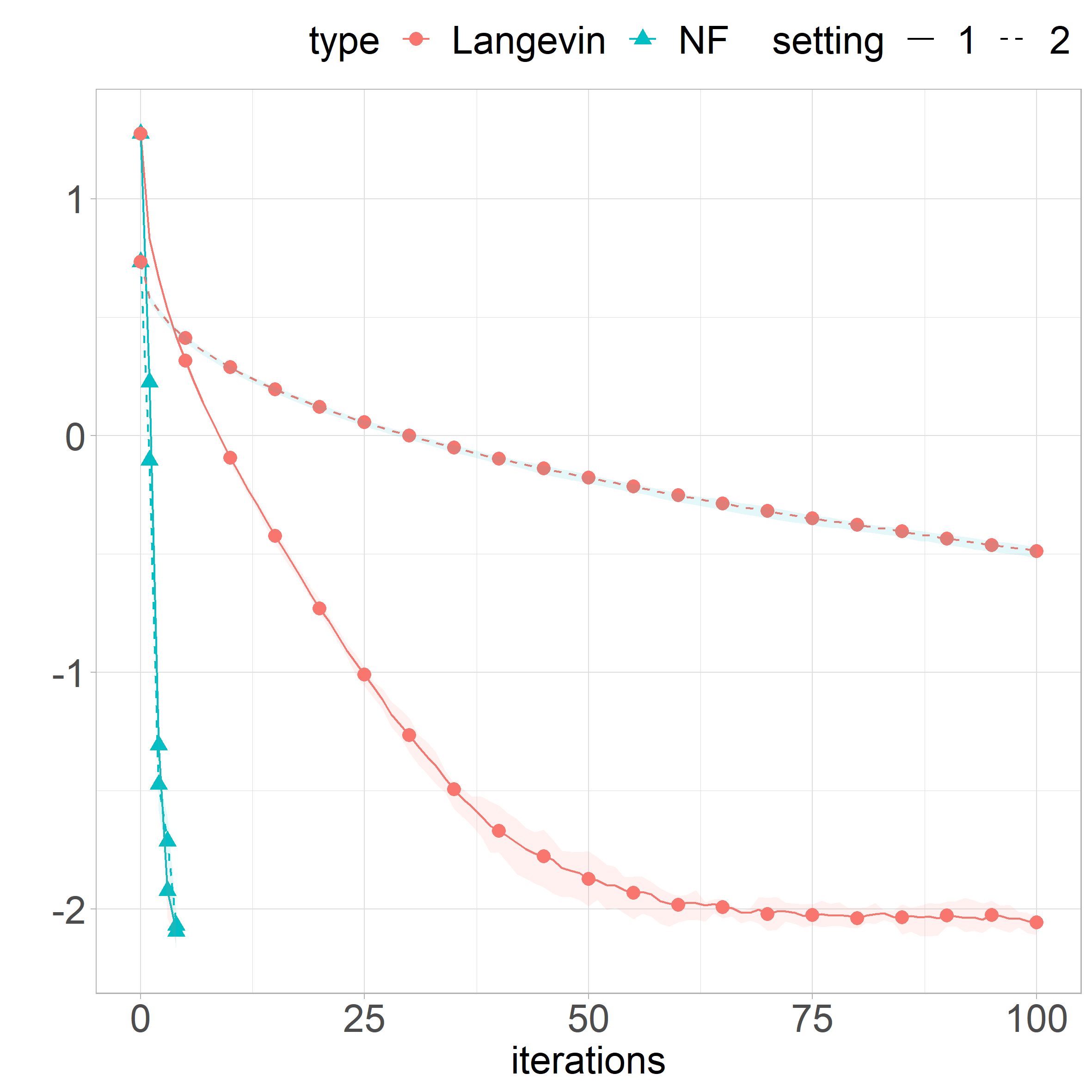}
      \caption{KL divergence}
   \end{subfigure}
   \caption{Numerical accuracy (with error bars) versus the iteration count $k$. For plots (a) and (b) (NPMLE), we report $\log\big(\m L_n(\rho_k) - \m L_n(\widehat\rho\,) \big)$, where $\widehat \rho\,$ is the (numerically) optimal solution; and for (c), we report $\log W_1(\rho_k,\,\pi)$, with $\pi$ denoting the target distribution. All results are based on $10$ independent trials.}\label{fig:convergence}
\end{figure*}

\noindent {\bf Gaussian location mixture model.}
We consider a two-dimensional Gaussian location mixture model, where for $\theta\in\Theta =\mb R^2$,  the conditional distribution $p(\,\cdot\,|\,\theta)$ in the NPMLE formulation~\eqref{eqn: NPMLE} is the density of $\m N(\theta, \, I_2)$. We set the true (mixing) distribution of $\theta$ to be a bimodal two moon distribution~\cite{Stimper2023}, and use a sample size of $n=5000$. For our method (\texttt{NF}), we also implement the stochastic variant (\texttt{NF\_s}) by using a randomly subsampled mini-batch of size $m=500$ to compute the stochastic gradient during the training of the normalizing flow.
We compare our method with the previously mentioned \texttt{WFR} and \texttt{KW} methods.
Figure~\ref{fig:convergence}(a) displays the difference between $\mathcal{L}_n(\rho_k) - \mathcal{L}_n(\widehat\rho\,)$ (in a logarithmic scale) as a function of the iteration count $k$, where $\widehat \rho$ is a numerically optimal solution obtained by running our method for a sufficient number of iterations. As can be seen, for this relatively simple problem, all methods exhibit rapid convergence. Our method with exact gradient descent (\texttt{NF}) achieves the fastest convergence, while our stochastic variant (\texttt{NF\_s}) shows slower convergence compared to WFR.

\noindent {\bf Gaussian location scale mixture model.}
Our second example is a $d$-dimensional Gaussian location scale mixture model, where for $\theta=(\mu,\,\sigma^2)\in \Theta=\mb R^d\times\mb R^d_{+}$, the conditional distribution $p(\,\cdot\,|\,\theta)$ is the density of $\m N(\mu, \Sigma)$, with $\Sigma = \diag(\sigma_1^2, \cdots, \sigma_d^2)$; and the true (mixing) distribution $P^\ast=P^\ast_{\mu}\otimes P^\ast_{\sigma^2}$. We consider two settings, one with $d=2$ and the other with $d=3$, both with a sample size of $n=5000$. Since \texttt{WFR} is applicable only to the Gaussian location mixture model, our comparison is limited to $\texttt{NF}$, $\texttt{NF\_s}$ and \texttt{KW}. Figure~\ref{fig:convergence}(b) shows the results. As we can observe, the necessity for KW to discretize the parameter space into equally spaced grids results in a non-vanishing bias term attributed to this discretization. This bias becomes larger as the dimensionality increases, owing to the curse of dimensionality. In contrast, our methods, including the stochastic variants, are relatively robust against dimensionality increase, with the numerical error keeps decreasing as the iteration count $k$ increases.

\noindent {\bf Bayesian sampling.}
In this example, we set the true target distribution to have density $\pi(\theta) \propto e^{-\frac{1}{2\alpha}\|\theta\|^{2\alpha}}$ for $\theta\in\Theta=\mb R^2$, which is known  only up to a normalization constant. The corresponding objective functional is $\m F(\rho) = \int \frac{1}{2\alpha}\|\theta\|^{2\alpha}\,\dd\rho(\theta) + \int\rho\log \rho$. We consider two settings: $\alpha=2$ and $\alpha=3$. Note that $\alpha=2$ corresponds to a Lipschitz continuous potential function $\frac{1}{2\alpha}\|\cdot\|^{2\alpha}$, as required by explicit discretization methods, while $\alpha=3$ violates this condition. We compare our method (\texttt{NF}) with the (unadjusted) \texttt{Langevin} method as a representative explicit discretization method. As illustrated in Figure~\ref{fig:convergence}(c), \texttt{NF} converges very rapidly for both values of $\alpha$, in line with the prediction of our Theorem~\ref{thm: discrete_regular} under $\lambda=1$. In contrast, \texttt{Langevin} exhibits significantly slower convergence, especially when $\alpha=3$. For \texttt{Langevin}, we manually selected the step size as the largest one that avoids divergence, ensuring the fastest convergence possible.

\noindent {\bf Impact of IKLPD step size $\tau_k$.}
To examine the impact of the step size $\tau_k$ on the IKLPD algorithm, we conduct additional numerical experiments to compare the inner loop iterations using a first-order optimization algorithm and the outer loop iterations of IKLPD under varying constant step sizes $\tau_k\equiv \tau$, using a stopping criterion based on the variance of the first variation. Due to space constraints, we defer the detailed results to Appendix~\ref{app:computation}.
From these results, we can see that for small (large) values of $\tau$, the sub-problem~\eqref{eqn: IKLPD} becomes easier (harder) as it needs fewer (more) inner loop iterations to meet the stopping criterion. However, since the progress made by each IKLPD step is smaller (larger), the total number of outer loop iterations increases (decreases). Furthermore, when $\tau$ surpasses a certain threshold, the inner loop fails to converge within a prescribed number of iterations. Upon closer examination of these non-converging cases, we identified two primary reasons for this failure, either the inner loop is trapped in a local minimum that is not global for problem~\eqref{eqn: IKLPD}, or it is unable to meet the stopping criterion within the prescribed number of iterations. This empirical observation aligns with our discussion following equation~\eqref{eqn: IKLPD}. 
We also defer a concluding discussion of this work to Appendix~\ref{app:dis}.

\bibliographystyle{plain}
\bibliography{ref}

\newpage
\appendix
\begin{center}
{\bf\Large Supplementary Materials: Appendix}
\end{center}
\numberwithin{equation}{section}

\newtheorem{theoremA}{Theorem}
\renewcommand{\thetheoremA}{A\arabic{theoremA}}
\newtheorem{lemmaA}[theoremA]{Lemma}
\renewcommand{\thelemmaA}{A\arabic{lemmaA}}

In this supplementary material, we provide more background knowledge and useful results about optimizing a functional over the space of all probability distributions, including some of their connections with optimal transport (e.g., displacement convexity) and a more broader framework of the proximal mirror descent algorithm that allows an extension from the KL divergence to general Bregman divergences. We also review additional literature on mirror descent and stochastic (proximal) mirror descent in the Euclidean space, along with some further optimization algorithms on the space of all probability distributions. Moreover, we detail the implementation of the algorithms and the numerical experiments showcased in the main paper, and we provide additional numerical results. Finally, this supplementary material includes all the proofs related to the main theoretical results presented in the paper, including the verification of the $L_2$-convexity of the target functionals in both the NPMLE and the Bayesian posterior computation examples.

\section{Backgrounds and Facts}
In this appendix, we provide additional background and facts related to optimization over the space of probability distributions, the non-parametric maximum likelihood estimation (NPMLE), and extensions of our developments to more general proximal mirror descent algorithms. However, as we mentioned in the main paper, our considered KL is often better aligned with the information geometry inherent to statistical problems; see, for example, the two motivating examples of NPMLE and Bayesian posterior computation considered in the paper.

\subsection{Some definitions and consequences}\label{app:more_def}
\paragraph{First variation.} We first provide a formal definition of first variations; more details can be found, e.g., in Section 7.2 of~\cite{santambrogio2015optimal}. Let $\m F:\ms P(\Theta)\to \mb R$ be a lower semi-continuous functional and $\ms P^r(\Theta)$ denote the set of all probability measures absolutely continuous with respect to the Lebesgue measure on $\Theta$. A measure $\rho\in\ms P(\Theta)$ is called \emph{regular} for $\m F$ if $\m F\big(\varepsilon\rho + (1-\varepsilon)\rho'\big) < \infty$ for all $\varepsilon\in(0, 1)$ and any $\rho'\in\ms P^r(\Theta)$ that has compact support and bounded density. If $\rho$ is regular for $\m F$, one can define the \emph{first variation} of $\m F$ at $\rho$ as a map $\frac{\delta\m F}{\delta\rho}(\rho): \Theta\to\mb R$ such that for any perturbation $\chi=\rho' - \rho$, where $\rho'\in\ms P^r(\Theta)$ has bounded density and compact support,
\begin{align*}
\frac{\dd}{\dd\varepsilon}\m F(\rho+\varepsilon\chi)\bigg|_{\varepsilon=0} = \int_{\Theta}\frac{\delta\m F}{\delta\rho}(\rho)\,\dd\chi.
\end{align*}

\paragraph{Pushforward.} Let $\m X$ be a measurable space, and $T:\m X\to\Theta$ be a measurable function. The \emph{pushforward} $\rho$ of a measure $\mu\in\ms P(\m X)$ under $T$, denoted by $\rho = T_\#\mu$, is a measure on $\Theta$ defined as
\begin{align*}
    \rho(A) = T_\#\rho(A) = \rho\big(T^{-1}(A)\big),\quad\forall A\subset \Theta\,\,\mx{is measurable}.
\end{align*}

\paragraph{$\mathbf{W_p}$ distance and coupling.}
Let $\pi_0, \pi_1: \Theta\times\Theta\to\Theta$ be the projection functions defined as $\pi_0(\theta, \theta') = \theta$ and $\pi_1(\theta, \theta') = \theta'$, and define $\pi_t = (1-t)\pi_0 + t\pi_1$. For any $\rho, \rho'\in\ms P(\Theta)$, $\gamma$ is called a \emph{coupling} of $\rho$ and $\rho'$, denoted by $\Pi(\rho, \rho')$, if $(\pi_0)_\#\gamma = \rho$ and $(\pi_1)_\#\gamma = \rho'$. Then, the $W_p$ distance between $\rho$ and $\rho'$ is defined as
\begin{align}\label{eqn: Wp-dist}
W_p^p(\rho, \rho') = \inf_{\gamma\in\Pi(\rho, \rho')} \int_{\Theta\times\Theta}\|\theta-\theta'\|^p\,\dd\gamma(\theta,\theta')
\end{align}
By the above definition, it is clear that the $W_p$ distance can also be defined through
\begin{align*}
W_p^p(\rho, \rho') = \inf_{\theta\sim\rho, \theta'\sim\rho'}\mb E\big[\|\theta-\theta'\|^p\big].
\end{align*}
We say $\gamma^\ast$ is an optimal coupling of $\rho$ and $\rho'$, denoted by $\Pi_o(\rho, \rho')$, if
\begin{align*}
W_2^2(\rho, \rho') = \int_{\Theta\times\Theta}\|\theta-\theta'\|^2\,\dd\gamma^\ast(\theta,\theta'),
\end{align*}
i.e. the infimum in~\eqref{eqn: Wp-dist} is achieved at $\gamma^\ast$. 

\paragraph{Wasserstein Geodesics, and (strong) convexity along geodesics.}
A (constant-speed Wasserstein) geodesics connecting $\rho_0$ and $\rho_1$ is a curve $\{\rho_t: 0\leq t\leq 1\}$ on $\ms P(\Theta)$, such that there exists $\gamma^\ast\in\Pi_o(\rho_0, \rho_1)$ satisfying $\rho_t = (\pi_t)_\#\gamma^\ast$. A functional $\m F$ is \emph{$\lambda$-strongly convex along geodesics} if
\begin{align*}
\m F(\rho_t) \leq (1-t)\m F(\rho_0) + t\m F(\rho_1) - \frac{\lambda}{2}\, t(1-t)\,W_2^2(\rho_0, \rho_1)
\end{align*}
holds for any geodesics $\{\rho_t: 0\leq t\leq 1\}$ and $t\in[0, 1]$.

\paragraph{Non-convexity along geodesics of NPMLE.}
Recall that given $n$ observations $X^n = (X_1, \cdots, X_n)$, NPMLE is defined as
\begin{equation}\label{eqn: NPMLE_problem}
\wht P_n = \argmin_{\rho\in\ms P(\Theta)}\m L_n(\rho),\quad\mx{where}\quad
\m L_n(\rho)\coloneqq -\frac{1}{n}\sum_{i=1}^n\log \Big(\int_\Theta p(X_i\,|\,\theta)\,\dd\rho(\theta)\Big) 
\end{equation}
The objective functional $\m L_n$ may not be geodesically convex. Consider $p(\cdot\,|\,\theta) = \m N(\theta, 1)$, $\rho_0 = \m N(0, 1)$ and $\rho_1 = \m N(0, 25)$. Since both $\rho_0$ and $\rho_1$ are Gaussian distributions, the optimal transport map from $\rho_0$ to $\rho_1$ is $T(\theta) = 5\theta$, and thus the geodescis connecting $\rho_0$ and $\rho_1$ is $\rho_t = \m N(0, (1+4t)^2)$. In this case, we have 
\begin{align*}
\m L_n(\rho_t) = \frac{\sum_{i=1}^nX_i^2}{2[1+(1+4t)^2]} + \frac{n}{2}\log\big[2\pi\big(1 + (1+4t)^2\big)\big].
\end{align*}
When $\rho^\ast = \delta_0$ is the point mass, $\frac{1}{n}\m L_n(\rho_t) \to \frac{1}{2}\log[2\pi(1+(1+4t)^2)] + \frac{1}{2[1+(1+4t)^2]}$ by law of large numbers. This function is not convex on $[0, 1]$. Similar result of non-convexity is numerically verified by~\cite{yan2023learning}.
\subsection{Extension from KL to general Bregman divergences}\label{app:general_divergence}
Let $\Phi:\ms P(\Theta)\to\mb R\cup\{+\infty\}$ be a ($L_2$-)convex functional with first variation $\frac{\delta\Phi}{\delta\rho}$. Define the associated Bregman divergence as
\begin{align*}
D_\Phi(\rho, \rho') \coloneqq \Phi(\rho) - \Phi(\rho') - \int_\Theta\frac{\delta\Phi}{\delta\rho}(\rho')\,\dd(\rho - \rho').
\end{align*}
Bregman divergence is always nonnegative due to the convexity of $\Phi$. In the implicit proximal mirror descent algorithm (with respect to the Bregman function $\Phi$) on the space of all probability distributions, given $\rho_0\in\ms P(\Theta)$ such that $\Phi(\rho_0)$ is finite, we iteratively solve 
\begin{align*}
\rho_k = \argmin_{\rho\in\ms P(\Theta)}\m F(\rho) + \frac{1}{\tau_k}D_\Phi(\rho, \rho_{k-1}), \quad k\geq 1.
\end{align*}
When $\Phi(\rho) = \int\rho\log\rho$ for $\rho\in\ms P^r(\Theta)$ and $\Phi(\rho) = +\infty$ for $\rho\notin\ms P^r(\Theta)$, it is easy to check that $D_\Phi(\rho, \rho') = \KL(\rho\,\|\,\rho')$.

\paragraph{$\chi^2$-divergence is not a Bregman divergence.} Recall that the $\chi^2$-divergence between two probability distributions are
\begin{align*}
\chi^2(\rho, \rho') = \int_\Theta \Big(\frac{\rho(\theta)}{\rho'(\theta)}-1\Big)^2\,\dd\rho(\theta).
\end{align*}
If there exists a convex functional $\Phi$, such that $\chi^2(\rho, \rho') = D_\Phi(\rho, \rho')$ for all $\rho, \rho'\in\ms P^r(\Theta)$. Let $\rho' = \m N(0, 1)$, we have
\begin{align*}
\Phi(\rho) &= \chi^2(\rho, \rho') + \Phi(\rho') + \int_\Theta\frac{\delta\Phi}{\delta\rho}(\rho')\,\dd(\rho-\rho')\\
&= \Phi(\rho') + \int_\Theta\frac{\delta\Phi}{\delta\rho}(\rho')\,\dd(\rho-\rho') + \int_\Theta\Big(\frac{\rho(\theta)}{\rho'(\theta)}-1\Big)^2\frac{\rho(\theta)}{\rho'(\theta)}\,\dd\rho'.
\end{align*}
Note that the first two terms are linear in $\rho$, while the last term is not convex with respect to $\rho$. Therefore, there is no convex functional $\Phi$ such that $\chi^2(\rho,\rho') = D_\Phi(\rho, \rho')$.

\paragraph{Renyi's $\alpha$-divergence is not a Bregman divergence.}
Recall that Renyi's $\alpha$-divergence~\cite{li2016renyi} is defined as
\begin{align*}
    R_\alpha(\rho, \rho') = \frac{1}{\alpha-1}\log\int_\Theta\rho(\theta)^\alpha\rho'(\theta)^{1-\alpha}\,\dd\theta, \alpha\in(0, 1).
\end{align*}
If there exists a convex functional $\Phi$, such that $R_\alpha(\rho,\rho') = D_\Phi(\rho, \rho')$ for all $\rho, \rho'\in\ms P^r(\Theta)$. Then
\begin{align}\label{eqn: Renyi}
\Phi(\rho) = \Phi(\rho') + \int_\Theta\frac{\delta\Phi}{\delta\rho}(\rho')\,\dd(\rho-\rho') + \frac{1}{\alpha-1}\log\int_\Theta\rho(\theta)^\alpha\rho'(\theta)^{1-\alpha}\,\dd\theta.
\end{align}
Taking the first variation on both sides of~\eqref{eqn: Renyi} yields
\begin{align*}
\frac{\delta\Phi}{\delta\rho}(\rho)(\theta) = \frac{\delta\Phi}{\delta\rho}(\rho')(\theta) + \frac{1}{\alpha-1}\cdot\frac{\alpha\rho(\theta)^{\alpha-1}\rho'(\theta)^{1-\alpha}}{\int_\Theta\rho(\theta)^\alpha\rho'(\theta)^{1-\alpha}\,\dd\theta}.
\end{align*}
Taking this expression of $\frac{\delta\Phi}{\delta\rho}(\rho)$ back to~\eqref{eqn: Renyi} yields contradiction.

\section{More Literature Review and Concluding Discussion}
In this appendix, we review more related literature and provide the deferred concluding discussion of this work.

\subsection{More related works}
\paragraph{Mirror descent.}
Mirror descent for convex optimization in the Euclidean space was originally proposed by~\cite{nemirovskij1983problem}. It is established that the mirror descent algorithm achieves a $O(k^{-1/2})$ convergence rate when dealing with a non-smooth convex objective function that possesses a uniformly bounded subgradient; this rate can be enhanced to $O(k^{-1})$ when the function is relatively smooth with respect to the Bregman divergence~\cite{bauschke2017descent}. When the objective function is convex and has Lipschitz gradients, \cite{krichene2015accelerated} demonstrates that the accelerated mirror descent converges at a rate of $O(k^{-2})$. For additional details on mirror descent algorithms in the Euclidean space, we refer the reader to the monographs \cite{beck2017first, lan2020first, wright2022optimization}.

\paragraph{Stochastic proximal (mirror) descent.}
Stochastic proximal descent type algorithms have been shown to be more stable than stochastic gradient type algorithms~\cite{ryu2014stochastic} when optimizing a function in the Euclidean space. However, they have been less extensively studied compared to the latter. Considering a scenario where the random objective function is restricted strongly convex, \cite{ryu2014stochastic} demonstrate that the expected $L_2$-distance between each iterate and the minima of the objective function converges exponentially fast, up to a constant factor.
In cases where the objective function is convex, \cite{asi2019stochastic} establish that the expected value of the objective function evaluated at each iterate approaches its global minimum at a polynomial rate. This is under the condition that the $L_2$-norm of the derivative of the stochastic objective function has uniformly bounded expected values. In contrast, \cite{bertsekas2011incremental} shows that in a bounded search space with a one-sided Lipschitz continuous objective function, the expected number of iterations needed to achieve $\varepsilon$-accuracy, up to a fixed constant, is of the order $O(\varepsilon^{-1})$. When it comes to stochastic proximal mirror descent, \cite{davis2018stochastic} prove a polynomial convergence rate for the expected value of the objective function across iterations, given a similar condition of one-sided Lipschitz continuity with respect to the square root of Bregman divergence.

\paragraph{Algorithms for optimizing functional on the space of probability distributions.}
Assuming the log-Sobolev inequality is satisfied, \cite{chizat2022mean} and \cite{nitanda2022convex} demonstrate an exponential convergence rate for minimizing the entropic regularized objective functional across the space of probabilities using mean-field Langevin dynamics. \cite{chizat2022mean} further show that the unregularized objective functional approaches its minimum at a rate of $O(\frac{\log\log t}{\log t})$, achieved by decreasing the regularization parameter at a rate of $O(\frac{1}{\log t})$ through an annealing argument. A similar annealing approach is employed in~\cite{chizat2022trajectory}, transforming trajectory inference problems into functional optimization problems. In a different vein, \cite{kent2021frank} introduce the Frank–Wolfe algorithm in the space of probabilities, inspired by distributionally robust optimization approaches.

\subsection{Concluding discussion of this work}\label{app:dis}



In this work, we proposed an implicit KL proximal descent (IKLPD) algorithm, which discretized a continuous-time gradient flow relative to the Kullback-Leibler divergence for minimizing a convex functional defined over the space of all probability distributions. We utilized the proposed method to address two statistical applications, specifically, non-parametric maximum likelihood estimation (NPMLE) and Bayesian posterior computation. We demonstrated that our implicit method has multiple advantages compared to its explicit counterpart: 1. it did not require a Lipschitz $L_2$-gradient, thus allowing for larger step sizes and fewer iterations to converge; 2. it was more robust and did not need kernel density estimation in order to approximately compute the $L_2$-gradient as in the explicit method, making the explicit method suffer from the curse of dimensionality. Computationally, we proposed a numerical method based on normalizing flow to implement IKLPD, and utilized a teacher-student architecture to maintain constant space complexity. Conversely, our numerical method could also be viewed as a novel approach that sequentially trains a normalizing flow for minimizing a convex functional with strong theoretical guarantees. Some potential future directions include: 1. applying and analyzing IKLPD for other more complicated statistical applications, such as training Bayesian neural networks and variational inference with structural constraints; 2. extending the KL to a general Bregman divergence and identifying examples where using a particular Bregman divergence is beneficial; 3. analyzing the optimization landscape of the normalizing flow for solving each implicit step optimization problem in the IKLPD.


\section{More Computational Details and Numerical Results}\label{app:computation}
In this appendix, we provide more details about our use of the normalizing flow for implementing the proposed IKLPD algorithm and the setup of the three numerical examples in the main paper. We also provide additional numerical results about: 1. the impact of step size $\tau_k$ on the inner/outer loop convergence of the IKLPD algorithm; 2. the teacher-student architecture for maintaining a fixed storage budget when composing short normalizing flows as the number of (outer loop) iterations increases. We conducted all experiments using the NVIDIA Tesla T4 GPU available on Google Colab.

\subsection{Implementation via normalizing flows}

Recall that the implicit KL proximal descent (IKLPD) algorithm minimizes the objective functional $\m F$ by iteratively solving the subproblem
\begin{align}\label{eqn: IKLPD_app}
\rho_{k} = \argmin_{\rho\in\ms P(\Theta)} \m F(\rho) + \frac{1}{\tau_k}\KL(\rho\,\|\,\rho_{k-1}), \,\, k\geq 1,
\end{align} 
with  an initialization $\rho_0\in\ms P^r(\Theta)$ and step size $\{\tau_k:\,k\geq 1\}$. The main idea of using normalizing flow (NF) to solve~\eqref{eqn: IKLPD_app} is to express $\rho_k$ through a map $T^{(k)}:\Theta\to\Theta$ and the initialization $\rho_0$, which can be easily sampled from, by letting $\rho_k = T^{(k)}_\#\rho_0$. A closed-form expression of the density of $\log \rho_k$ derived from the normalizing flow enables exact computation of $\KL(\rho_k\,\|\,\rho_{k-1})$ through $T^{(k)}$ and $T^{(k-1)}$. Specifically, when the map $T^{(k)}$ is invertible and differentiable (which is satisfied by NF), if we denote the Jacobian matrix of $T^{(k)}$ by $J_{T^{(k)}}$, then the change of variable formula implies
\begin{align}\label{eqn: change_of_variable}
\rho_k(\theta) = \rho_0\big((T^{(k)})^{-1}(\theta)\big)\big|\det J_{T^{(k)}}\big((T^{(k)})^{-1}(\theta)\big)\big|^{-1}
\end{align}
In practice, the reparametrization trick can be employed to simplify the numerical computation. Concretely, let $\tilde\rho_k$ be empirical distribution of $M$ particles $\theta_1^{(k)}, \cdots, \theta_M^{(k)}$ sampled from $\rho_k = T^{(k)}_\#\rho_0$. By applying~\eqref{eqn: change_of_variable}, the objective functional in~\eqref{eqn: IKLPD_app} can be approximated by
\begin{align}
\begin{aligned}\label{eqn: subproblem}
\m F_k \coloneqq \m F(\tilde\rho_k) + \frac{1}{M\tau_k}\sum_{j=1}^M \bigg[\log\frac{\rho_0\big((T^{(k)})^{-1}(\theta_j^{(k)})\big)}{\rho_0\big((T^{(k-1)})^{-1}(\theta_j^{(k)})\big)} - \log\frac{\big|\det J_{T^{(k)}}\big((T^{(k)})^{-1}(\theta_j^{(k)})\big)\big|}{\big|\det J_{T^{(k-1)}}\big((T^{(k-1)})^{-1}(\theta_j^{(k)})\big)\big|}\bigg].
\end{aligned}
\end{align}
Computing~\eqref{eqn: subproblem} requires efficient computation of the inverse maps of $T^{(k)}$ and $T^{(k-1)}$, which makes NF an appropriate choice for modeling these maps. An NF model with length $L$ is a map composed of $L$ invertible transformations $T_1,\cdots, T_L$, the inverse of which can be easily calculated. The invertibility of the NF model is guaranteed by the invertibility of these transformations.
We choose the NF model with Real-NVP architecture~\cite{dinh2016density}, where the transformations $\{T_l: 1\leq l\leq L\}$ are affine coupling blocks.
See Algorithm~\ref{alg:cap1} for a summary of this straightforward implementation of IKLPD using NF via the Adam optimizer. In Appendix~\ref{app:composition_NF} below, we present a computationally efficient method for sequentially stacking local, short normalizing flows, learned within each inner-loop iteration, to form a global, layered normalizing flow for approximating the target solution $\rho^\ast$. Additionally, we conduct a numerical experiment to compare this compositional scheme via a teacher-student architecture with Algorithm~\ref{alg:cap1}, which re-trains a long normalizing flow for each subproblem.

\begin{algorithm}[t]
\caption{Implementing IKLPD with Normalizing Flows}\label{alg:cap1}
\begin{algorithmic}
\Require data $X^n = (X_1, \cdots, X_n)$; initialized NF model $T^{(0)}$; number of particles $M$; initialization $\rho_0$; learning rate of Adam optimizer $\gamma$; step size $\tau$; number of outer iterations $N_1$; number of inner iterations $N_2$; the decay factor of the learning rate in Adam $\beta_1$;  the increase factor of the step size $\beta_2$
\algrule
\State Sample $M$ particles $\underaccent{\tilde}{\theta}^{(0)} = [\theta^{(0)}_1, \cdots, \theta^{(0)}_M]$ from $\rho_0$.
\For{$k$ = $1$ to $N_1$}
    \State Initialize $T^{(k)}$ as $T^{(k-1)}$.
    \State Compute the learning rate $\gamma_k = \gamma\cdot \beta_1^{k-1}$.
    \State Compute the step size $\tau_k = \tau\cdot \beta_2^{k-1}$.
    \For{$r$ = $1$ to $N_2$}
        \State $\underaccent{\tilde}{\theta}^{(k)} = T^{(k)}(\underaccent{\tilde}{\theta}^{(0)})$.
        \State Compute the loss $\m F_k = \m F_k(\underaccent{\tilde}{\theta}^{(k)}, \tau_k)$ in~\eqref{eqn: subproblem} with $\underaccent{\tilde}{\theta}^{(k)}$.
        \State Update $T^{(k)}$ based on the loss $\m F_k$ using Adam optimizer with learning rate $\gamma_k$.
    \EndFor
\EndFor
\end{algorithmic}
\end{algorithm}

\subsection{More implementation details of three examples in the paper}\label{app: experiment_details}

\paragraph{Gaussian location mixture model.}
We consider a two-dimensional Gaussian location mixture model with the parameter space $\Theta = \mb R^2$. The true distribution $P^\ast$ of the parameter $\theta$ is a bimodal two moon distribution~\cite{Stimper2023}. The conditional distribution in NPMLE is $p(\cdot\,|\,\theta) = \m N(\theta, I_2)$, and $n=5000$ samples are generated from the model
\begin{align}\label{eqn: data_generate}
\theta_i \stackrel{\rm{i.i.d.}}{\sim} P^\ast \quad\mx{and}\quad X_i\,|\,\theta_i\sim\m N(\theta_i, I_2), \quad i=1, 2, \cdots, n.
\end{align}
In our method, the NF model consists of $30$ affine coupling blocks and each block contains two hidden layers with width $256$. The initialization is $\rho_0=\m N(0, 4I_2)$, and $M=3000$ particles are generated to approximate the probability measure $\rho_k$ in each iteration. The outer iteration is run $N_1 = 25$ times with the step size $\tau_k=5 \times \beta_2^{k-1}$ where the increase factor is $\beta_2 = 1.15$. In the $k$-th outer iteration, the subproblem~\eqref{eqn: subproblem} with $\m F = \m L_n$ defined in~\eqref{eqn: NPMLE_problem} is optimized via Adam optimizer with the initialized learning rate $\gamma_k = 10^{-4}\beta_1^{k-1}$ and the rate decay factor $\beta_1 = 0.912$ for at most $N_2 = 1000$ inner iterations. The inner loop stops early if the $L_2$-norm of the gradient of the parameters in the NF model $T^{(k)}$ reaches the threshold $10^{-4}$ or stops decreasing for $200$ consecutive inner iterations.

In the stochastic variant of NF, a randomly subsampled mini-batch of size $m = 500$ from samples $X^n$ is used to compute the stochastic gradient during the training of the NF. Different from the deterministic NF, the increase factor is $\beta_2 = 1$, and the initialized learning rate in Adam optimizer is $1/(1+k/27)$, which decays along outer iterations. All other settings are same as the ones in the deterministic NF model.

In this experiment, our methods are compared with the KW method~\cite{koenker2014convex} and the Wasserstein--Fisher--Rao (WFR) gradient flow~\cite{yan2023learning}. In the KW method, the probability measure is approximated by a discrete probability distribution supported on a fixed grid. Each grid point can be viewed as a particle with a fixed location, and the goal is to minimize $\m L_n$ by finding the optimal weights of these particles, which can be achieved by applying Algorithm 2 in~\cite{yan2023learning}; this algorithm updates the weights of the particles by explicitly discretizing the Fisher--Rao gradient flow. On the other hand, both the locations and the weights are updated in WFR method by discretizing the WFR gradient flow through particles.

In both of these two methods, the step size is $\tau=1$, as it is the largest step size to ensure that these methods converge. In the KW method, by letting $L = \| X^n \|_\infty$, probability distributions are approximated by a discrete probability distribution supported on a fixed and equally spaced grid on $[-L,L]^2$ with total $3025$ grid points, and the mass on each grid is updated to minimize the functional loss $\m L_n$ via Algorithm 2 in~\cite{yan2023learning}. In the WFR method, we directly use Algorithm 1 in~\cite{yan2023learning} with the same initialization $\rho_0$ and the number of particles $M$ as in our method.

\paragraph{Gaussian location scale mixture model.}
We consider a $d$-dimensional Gaussian location scale mixture model with parameters $\theta = (\mu, \sigma^2)\in\Theta=\mb R^d\times\mb R_+^d$. The conditional distribution is $p(\cdot\,|\,\theta) = \m N(\mu, \Sigma)$ with $\Sigma = \diag(\sigma_1^2, \cdots, \sigma_d^2)$, and the true joint mixing distribution is $P^\ast = P_\mu^\ast \otimes P_{\sigma^2}^\ast$. We consider two settings. In Setting 1, we let $d=2$ and $P_\mu^\ast$ be the bimodal two moon distribution. In Setting 2, we let $d=3$ and $P_\mu^\ast$ be the tensor product of a bimodal two moon distribution for the first two coordinates of $\mu$ and a standard normal distribution for the last coordinate of $\mu$. In both settings, we set $P_{\sigma^2}^\ast$ to be the joint distribution of $d$ independent $\chi^2$ distributions with degree of freedom $1$, and the sample size is $n=5000$. 

In our methods, we use an NF model with $2d$-dimensional inputs and outputs, where the first $d$ dimensions represent location parameters and the last $d$-dimensions represent scale parameters. The NF model consists of $30$ affine coupling blocks and each block contains two hidden layers with width $64$. The initialization is $\rho_0=\m N(0, 4I_d)$. $M$ particles are generated to approximate the probability measure $\rho_k$ in each iteration, and we choose $M=2041$ in Setting 1 and $M=4096$ in Setting 2. The outer iteration is run $N_1 = 50$ times. All other hyperparameters and the stopping criterion of the inner loop in deterministic NF and stochastic NF are same as in the experiments of Gaussian location mixture models.

In the experiment, our methods are compared with the KW method. By letting $L = \| X^n \|_\infty$, probability distributions are approximated by a discrete probability distribution supported on a fixed and equally spaced grid on $[-L, L]^d \times [0.01, 4]^d$ with total $M=2041$ grid points in Setting 1 and $M=4096$ grid points in Setting 2. The mass on each grid is updated to minimize the functional loss $\m L_n$ via Algorithm 2 in~\cite{yan2023learning} with step size $1$.

\paragraph{Bayesian posterior sampling.} 
The goal is to minimize the KL divergence $\m F(\rho) = \KL(\rho\,\|\,\pi)$, where the target distribution $\pi(\theta)\propto e^{-\frac{1}{2\alpha}\|\theta\|^{2\alpha}}$ is known up to a normalization constant and $\theta\in\Theta = \mb R^2$. We consider two settings with $\alpha=2$ in Setting 1 and $\alpha=3$ in Setting 2. In our method, the NF model consists of $20$ affine coupling blocks and each block contains two hidden layers with width $64$. With initialization $\rho_0=\m N(0, 9I_2)$ in Setting 1 and $\rho_0=\m N(0, 4I_2)$ in Setting 2, $M=1000$ particles are generated to approximate the probability measure $\rho_k$ in each iteration. The outer iteration is run $N_1 = 25$ times with the step size $\tau_k=5$ for all $k\geq 1$ (i.e. the increase factor is $\beta_2 = 1$). In the $k$-th outer iteration, the subproblem~\eqref{eqn: subproblem} with $\m F = \KL(\cdot\,\|\,\pi)$ is optimized via Adam optimizer with the initialized learning rate $\gamma_k = 10^{-4}\beta_1^{k-1}$ and the rate decay factor $\beta_1 = 0.912$ for $N_2 = 1000$ inner iterations. Our method is compared with Langevin dynamics, where $M=1000$ particles are generated from the same initialization as in our method and updated by an explicit discretization of Langevin dynamics,
\begin{align*}
\theta_j^{(k)} = \theta_j^{(k-1)} - \Delta t \big\| \theta_j^{(k-1)} \big\|^{(2\alpha - 2)} \theta_j^{(k-1)} + \sqrt{2 \Delta t}\, u_j^{(k)}, \quad j=1,\cdots, M,
\end{align*}
where $u_j^{(k)}$ are i.i.d. samples generated from $\m N(0, I_2)$. We set $\Delta t = 10^{-2}$ in Setting 1 and $\Delta t = 4 \cdot 10^{-4}$ in Setting 2, as they are the largest $\Delta t$ to ensure that the discretized Langevin dynamics does not diverge and therefore leads to the fastest convergence possible.

\subsection{Impact of IKLPD step size $\tau_k$}\label{app: tau_impact}
Recall that the first variation of $\m L_n$ in the NPMLE problem~\eqref{eqn: NPMLE_problem} at a probability measure $\rho$ is the map
\begin{align} \label{eqn: npmle_first_variation}
    \frac{\delta \m L_n }{\delta \rho}(\rho): \theta \mapsto - \frac{1}{n} \sum^n_{i=1} \frac{p(X_i\,|\,\theta)}{ \int_\Theta p(X_i\,|\,\theta)\,\dd\rho(\theta) }.
\end{align}
For any (local) minimum $\rho$  of $\m L_n$, the first-order optimality condition (FOC) implies that $\frac{\delta\m L_n}{\delta\rho}(\rho)$ is a constant on the support of $\rho$ almost everywhere. In the experiments, we use the variance of first variation $\var_{\theta\sim\rho}\big(\frac{\delta\m L_n}{\delta\rho}(\rho)(\theta)\big)$ to characterize the closeness of $\frac{\delta\m L_n}{\delta\rho}(\rho)$ to a constant. In the $k$-th iteration, this variance at $\rho_k$ can be approximated by the sample variance of
\begin{align*} 
    \bigg\{ - \frac{1}{n} \sum^n_{i=1} \frac{p(X_i\,|\,\theta_j^{(k)})}{ \frac{1}{M} \sum^M_{j=1} p(X_i\,|\,\theta_j^{(k)}) } : 1\leq j\leq M \bigg\}
\end{align*}
given $M$ particles $\theta_1^{(k)}, \cdots, \theta_M^{(k)}$ generated from $\rho_k$. When the sample variance is smaller than a threshold $\zeta$ at some iteration $k$, we choose $\rho_k$ as the final solution of the NPMLE problem.

Similarly, since the first variation $\m L_n(\rho) + \frac{1}{\tau_k}\KL(\rho\,\|\,\rho_{k-1})$ is
\begin{align*}
-\frac{1}{n}\sum_{i=1}^n\frac{p(X_i\,|\,\theta)}{\int_\Theta p(X_i\,|\,\theta)\,\dd\rho(\theta)} + \frac{1}{\tau_k}\log\frac{\rho(\theta)}{\rho_{k-1}(\theta)},
\end{align*}
the variance of this first variation of the subproblem at $\rho_k$ can be approximated by the sample variance of
\begin{align*} 
    \bigg\{ - \frac{1}{n} \sum^n_{i=1} \frac{p(X_i\,|\,\theta_j^{(k)})}{ \frac{1}{M} \sum^M_{j=1} p(X_i\,|\,\theta_j^{(k)}) } + \frac{1}{\tau_k} \log\bigg( \frac{\rho(\theta_j^{(k)})}{\rho_{k-1} (\theta_j^{(k)})} \bigg) : 1\leq j\leq M \bigg\}.
\end{align*}
When this sample variance is smaller than a threshold $\zeta_k$, the inner loop stops and the current $T^{(k)}$ is used to construct the solution $\rho_k$ of the subproblem through $\rho_k = T^{(k)}_\#\rho_0$.

\begin{figure}[t]
\centering
     \includegraphics[scale = 0.45]{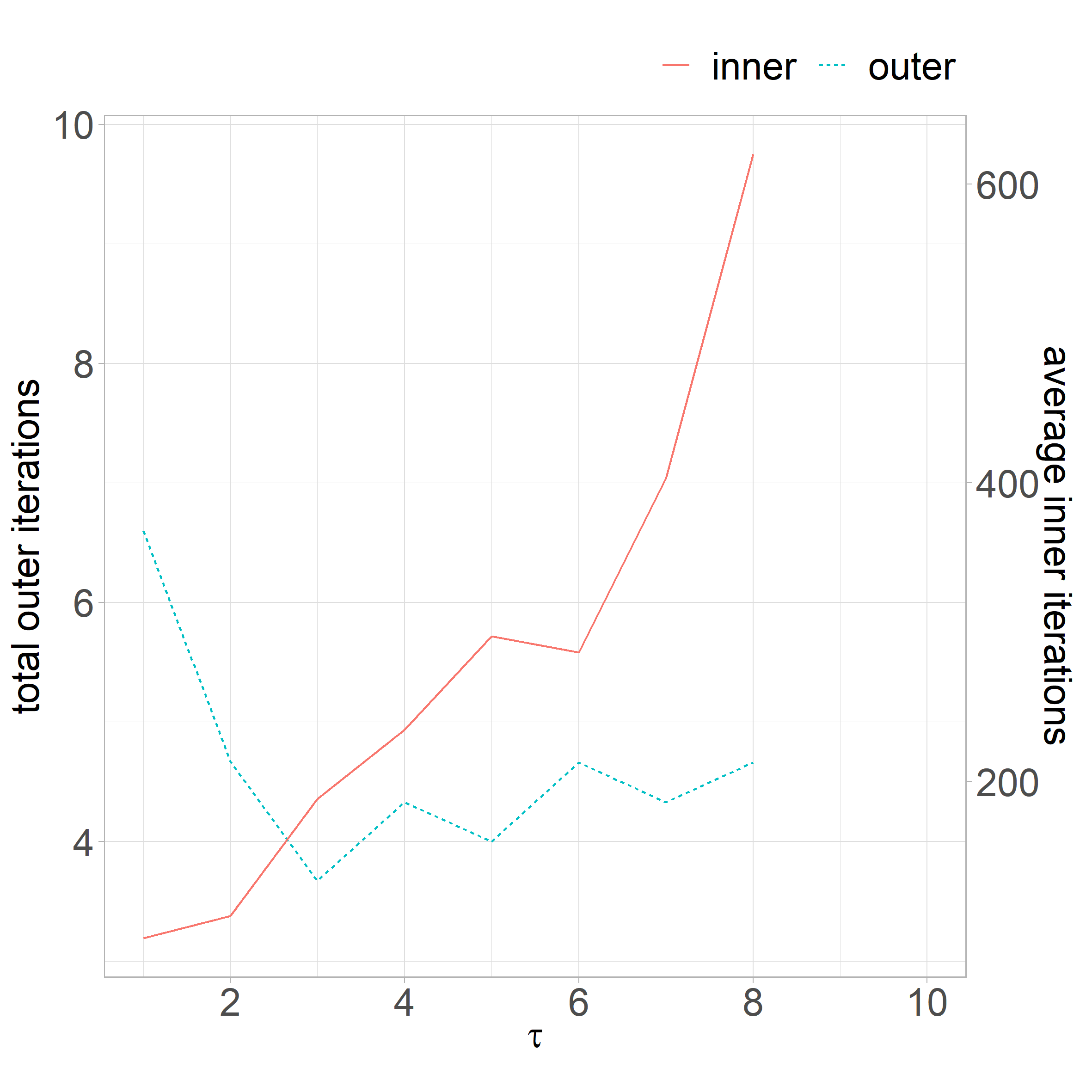}
   \caption{Outer and inner loop iterations (over $5$ independent trials) versus step size $\tau$. Here, we designate the first two outer iterations as a burn-in period, applying the stopping criterion only after this burn-in; moreover, we include the burn-in period in the total count of outer iterations, resulting in $3$ as the smallest possible number of outer iterations.
   As we can observe, the implicit step optimization problem (\ref{eqn: subproblem}) becomes easier as the step size $\tau$ decreases. This trend is reflected in the lower average number of inner iterations and higher number of outer iterations when the step size $\tau$ is smaller.  We note that since the minimal outer loop iteration is $3$, the (averaged) outer loop iterations tend to stabilize within the interval $[3,5]$ for those relatively large $\tau$ values with convergent inner loop iterations. However, when $\tau$ exceeds the threshold $8$, the inner loop is unable to converge within a prescribed number of $5000$ iterations. For these instances, we have chosen not to plot the corresponding inner and outer iterations.}
   \label{fig: tau_impact}
\end{figure}

Figure~\ref{fig: tau_impact} summarizes our numerical results, illustrating the impact of the step size $\tau_k$ on the IKLPD algorithm. Here, we report the number of inner loop iterations executed using the Adam optimizer and the outer loop iterations of IKLPD under various constant step sizes $\tau_k\equiv \tau$, employing the stopping criterion based on the aforementioned variance of the first variation. Note that in the implementation, we designate the first two outer iterations as a burn-in period, applying the stopping criterion only after this burn-in; moreover, we include the burn-in period in the total count of outer iterations, resulting in $3$ as the smallest possible number of outer iterations. We observe that for small (large) $\tau$ values, the sub-problem~\eqref{eqn: IKLPD} becomes simpler (more complex), requiring fewer (more) inner loop iterations to meet the stopping criterion. However, since each IKLPD step results in smaller (larger) progress, the total number of outer loop iterations correspondingly rises (falls). We note that since the minimal outer loop iteration is $3$, the (averaged) outer loop iterations tend to stabilize within the interval $[3,5]$ for those relatively large $\tau$ values with convergent inner loop iterations.
Additionally, when $\tau$ exceeds a particular threshold (which is $8$), the inner loop does not converge within the prescribed upper limit of $5000$ iterations. Overall, $\tau=3$ seems to be the optimal step size that balances the inner and outer loop convergences for this particular example. This empirical finding is consistent with the discussion that follows equation~\eqref{eqn: IKLPD}.

The detailed experiment setup for this numerical study is described as follows.
We consider a similar experiment setting as the Gaussian location mixture model. The true mixing distribution $P^\ast$ is a bimodal two moon distribution with $1.4$ times larger distance between two modes than the setting in Appendix~\ref{app: experiment_details}. $n=1000$ samples are generated through the data generating process~\eqref{eqn: data_generate}, and the step size $\tau_k=\tau$ is a constant. We use an NF model consisting of $10$ affine coupling blocks, and each block contains two hidden layers with width $64$. The initialization is $\rho_0 = \m N(0, I_2)$, and $M=1000$ particles are generated to approximate the probability distribution $\rho_k$ at each iteration. 

The threshold for the outer iterations is $\zeta = 0.05$. In the $k$-th outer iteration, the threshold for the inner loop is $\zeta_k = \frac{0.07 \cdot 20}{19 + k}$. If this convergence condition of the inner loop is not met within $5000$ iterations at the $k$-th iteration, we claim that the NF model fails find $\rho_k$ due to a overly large choice of the step size $\tau$. The maximum outer iteration is $50$. We set the first two outer iterations as burn-in iterations, where the NF model will not be considered to fail to converge for the first two outer iterations if the convergence condition is not met. 

We select the mirror step size $\tau \in\{1,2,\ldots,9,10\}$. For each $\tau$, the learning rate $\gamma$ for the Adam optimizer is selected to make the algorithm converge in smallest number of outer iteration. In the $k$-th outer iteration, the Adam optimizer with the initialized learing rate is $\gamma_k = \frac{20 \cdot \gamma}{19 + k}$, which decays along outer iterations. When $\tau = 9$ or $10$, for various choices of learning rate $\gamma$, the NF model fails to find $\rho_k$ at some iteration $k$ after the burn-in period.

\subsection{Composition of short flows and teacher-student architecture}\label{app:composition_NF}

\begin{algorithm}[t]
\caption{IKLPD with composition of short flows and the teacher-student architecture}\label{alg:cap2}
\begin{algorithmic}
\Require data $X^n = (X_1, \cdots, X_n)$; initialized NF model $T^{(0)}$; number of particles $M$; initialization $\rho_0$; learning rate of Adam optimizer $\gamma$; learning rate of Adam optimizer used in the compression process $\gamma'$; step size $\tau$; number of outer iterations $N_1$; number of inner iterations $N_2$; number of compression iterations $N_3$; the length of the compressed flow $k_0$; the maximum length of the flow before compression $k_1$; the length of the short flow $k_2$; the compression $L_2$ loss threshold $\epsilon$; the decay factor of the learning rate in Adam $\beta_1$;  the increase factor of the step size $\beta_2$
\algrule
\State Sample $M$ particles $\underaccent{\tilde}{\theta}^{(0)} = [\theta^{(0)}_1, \cdots, \theta^{(0)}_M]$ from $\rho_0$.
\For{$k$ = $1$ to $N_1$}
    \State Initialize a short flow $T'^{(k)}$ with length $k_2$ such that $T'^{(k)}_\# \rho_{k-1} = \rho_{k-1}$.
    \State Set RequiresGrad = False for all parameters in $T^{(k-1)}$. \Comment{the parameters in $T^{(k-1)}$ are fixed}
    \State Initialize $T^{(k)}$ as $T^{(k-1)} \circ T'^{(k)}$.  \Comment{only the parameters in $T'^{(k)}$ could be learned}
    \State Compute the learning rate $\gamma_k = \gamma\cdot \beta_1^{k-1}$.
    \State Compute the step size $\tau_k = \tau\cdot \beta_2^{k-1}$.
    \For{$r$ = $1$ to $N_2$}
        \State $\underaccent{\tilde}{\theta}^{(k)} = T^{(k)}(\underaccent{\tilde}{\theta}^{(0)})$.
        \State Compute the loss $\m F_k = \m F_k(\underaccent{\tilde}{\theta}^{(k)}, \tau_k)$ in~\eqref{eqn: subproblem} with $\underaccent{\tilde}{\theta}^{(k)}$.
        \State Update $T^{(k)}$ based on the loss $\m F_k$ using Adam optimizer with learning rate $\gamma_k$.
    \EndFor
\If{Length of the NF model $T^{(k)} > k_1$}
    \State Initialize a flow $T''^{(k)}$  with length $k_0$ as $T''^{(k)}_\# \rho_0 = \rho_0$.
    \For {$s$ = $1$ to $N_3$}:
         \State Compute the $L_2$ loss $L_2(T^{(k)}, T''^{(k)} ) \coloneqq \ell_2(T^{(k)}(\underaccent{\tilde}{\theta}^{(0)}), T''^{(k)}(\underaccent{\tilde}{\theta}^{(0)}))$.
        \State Update $T''^{(k)}$ based on $L_2(T^{(k)}, T''^{(k)} )$ using Adam optimizer with learning rate $\gamma'$.
        \If{$L_2(T^{(k)}, T''^{(k)} )\leq \epsilon$}
            \State Break the current loop.
        \EndIf
    \EndFor
    \State Let $T^{(k)} = T''^{(k)}$.
\EndIf

\EndFor
\end{algorithmic}
\end{algorithm}

\begin{figure}[t]
\centering
     \includegraphics[scale = 0.45]{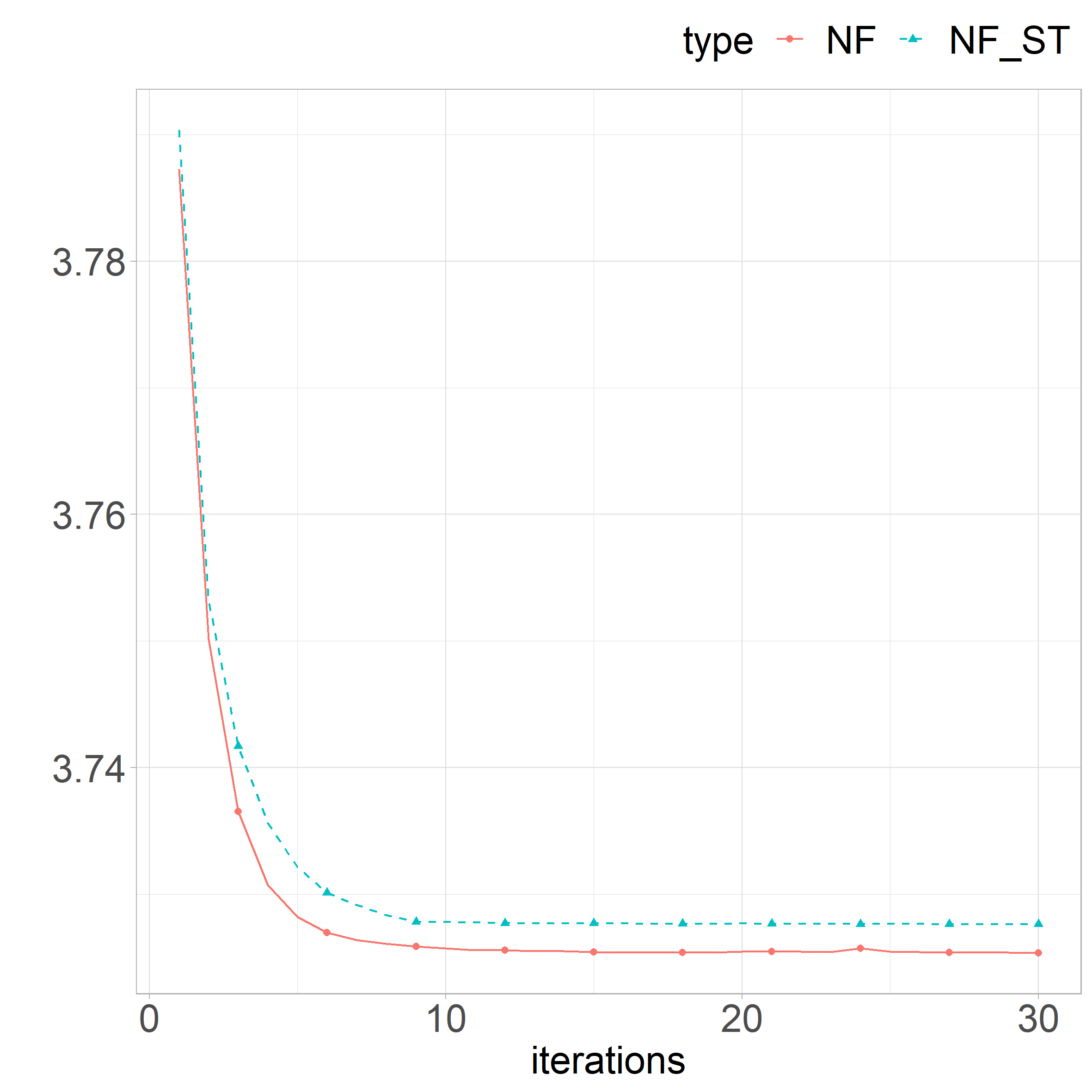}
   \caption{Optimized objective value $\m L_n(\rho_k)$ versus iteration count $k$. We report the averaged $\m L_n (\rho_k)$ over $5$ independent trials. The results indicate that \texttt{NF\_{ST}} and \texttt{NF} demonstrate very similar performance. The compression through the teacher-student architecture successfully maintains the expressive capability of the original model.}\label{fig:TS}
\end{figure}

Algorithm \ref{alg:cap2} summarizes the algorithm for the compositional scheme of sequentially stacking the local, short normalizing flows (each with length $k_2$), learned within each inner-loop iteration, to
form a global, layered normalizing flow for minimizing $\m F$. In the $k$-th outer iteration, the total length of the large NF model is $(k+1)\cdot k_2$ since we composite a new length-$k_2$ short flow with the original NF model with length $k\cdot k_2$. When the length of this compositional NF model exceeds the threshold of maximum length $k_1$, we employ a teacher-student architecture to distill knowledge from the compositional NF model to a shorter NF model of length $k_0$. This is achieved by minimizing (a sample version of) the $L_2$ distance between the larger (teacher) NF model and the smaller, length-$k_0$ (student) NF model.

Figure \ref{fig:TS} provides a numerical comparison between Algorithm~\ref{alg:cap2} that re-trains a long normalizing flow for each subproblem (indicated as \texttt{NF}) and Algorithm \ref{alg:cap2} that uses the compositional scheme and teacher-student architecture (indicated as \texttt{NF\_{ST}}). As we can see, the expressive capability of the composited normalizing flow model is comparable to that of the computationally more expensive \texttt{NF} method, which re-trains a lengthy normalizing flow at each iteration of the IKLPD algorithm. In addition, by employing the teacher-student architecture, we can preserve a constant storage budget while maintaining the expressive capability of the compositional normalizing flow model.

We describe below the concrete setting of this numerical experiment.
We use a similar objective functional as in the Gaussian location mixture models. The NF model consists of $30$ affine coupling blocks, and each block contains two hidden layers having $256$ units. The step size $\tau_k = \tau$ is fixed (i.e. the increase factor $\beta_2 = 1$), and no early stopping criterion is applied for the inner loop. All other hyperparameters are the same as in the deterministic NF model in Appendix~\ref{app: experiment_details}.

Algorithm~\ref{alg:cap1} is compared with the composition of short NF model with teacher-student architecture (\texttt{NF\_{ST}}) as shown in Algorithm~\ref{alg:cap2}. In each outer iteration, a short flow with length $k_2 = 4$ is composited with the original flow. Each short flow consists of $4$ affine coupling blocks and each block contains two hidden layers with width $512$. When the length of the composited flow exceeds the maximum length $k_1 = 40$, it will be compressed into a flow with length $k_0 = 20$. The initialized Adam learning rate is $\gamma = 8\times 10^{-5}$. The compression process is run at most $N_3 = 3000$ iterations with the Adam learning rate $\gamma' = 10^{-5}$ and stops early if the $L_2$ distance between the compressed flow and the original composited flow is less than the threshold $\epsilon = 10^{-4}$. Other hyperparameters are the same as the NF model in Appendix~\ref{app: experiment_details}

\allowdisplaybreaks

\section{Proofs of Theoretical Results}\label{app: proof}
In this appendix, we provide all deferred proofs for the main theoretical results from the main paper.

\subsection{Proof of Theorem~\ref{thm: FR_dynamics}}\label{app: cts_convergence}
Taking derivative with respect to time yields
\begin{align}
\frac{\dd}{\dd t}\KL(\rho^\ast\,\|\,\rho_t) 
&= -\frac{\dd}{\dd t}\int\log\rho_t\,\dd\rho^\ast\nonumber\\
&= \int_\Theta \Big(\frac{\delta\m F}{\delta\rho}(\rho_t)(\theta) - \int_\Theta \frac{\delta\m F}{\delta\rho}(\rho_t)(\theta')\,\dd\rho_t(\theta')\Big)\,\dd\rho^\ast(\theta)\nonumber\\
&= \int \frac{\delta\m F}{\delta\rho}(\rho_t)\,\dd(\rho^\ast - \rho_t)(\theta)\nonumber\\
&\leq \m F(\rho^\ast) - \m F(\rho_t)-\frac{\lambda}{2}\KL(\rho^\ast\,\|\,\rho_t).\label{eqn: KL_derivative}
\end{align}
When $\lambda > 0$, since $\m F(\rho^\ast) - \m F(\rho_t) \leq 0$, we have
\begin{align*}
\frac{\dd}{\dd t}\KL(\rho^\ast\,\|\,\rho_t) \leq -\frac{\lambda}{2}\KL(\rho^\ast\,\|\,\rho_t).
\end{align*}
By Gronwall's inequality, we have
\begin{align*}
    \KL(\rho^\ast\,\|\,\rho_t) \leq e^{-\frac{\lambda t}{2}}\KL(\rho^\ast\,\|\,\rho_0).
\end{align*}
When $\lambda = 0$,~\eqref{eqn: KL_derivative} is equivalent to
\begin{align*}
\frac{\dd}{\dd t}\KL(\rho^\ast\,\|\,\rho_t) \leq \m F(\rho^\ast) - \m F(\rho_t).
\end{align*}
This implies
\begin{align*}
\KL(\rho^\ast\,\|\,\rho_t) - \KL(\rho^\ast\,\|\,\rho_0) \leq t\m F(\rho^\ast) - \int_0^t\m F(\rho_s)\,\dd s.
\end{align*}
By Jensen's inequality, we have
\begin{align*}
\m F\Big(\frac{1}{t}\int_0^t \rho_s\,\dd s\Big) - \m F(\rho^\ast) \leq \frac{1}{t}\int_0^t \m F(\rho_s)\,\dd s - \m F(\rho^\ast) \leq \frac{1}{t}\KL(\rho^\ast\,\|\,\rho_0).
\end{align*}

\subsection{Proof of Theorem~\ref{thm: discrete_regular}}
We need the following lemma to bound the functional value at each iterate, the proof of which is deferred to Section~\ref{app: lemmas} in this supplementary file. 
\begin{lemmaA}\label{lem: one_step_improve}
For any $\rho\in\ms P^r(\Theta)$
\begin{align*}
\m F(\rho_{k}) - \m F(\rho) \leq \frac{1}{\tau_k}\KL(\rho\,\|\,\rho_{k-1}) - \Big(\frac{1}{\tau_k} + \frac{\lambda}{2}\Big) \KL(\rho\,\|\,\rho_{k})  - \frac{1}{\tau_k} \KL(\rho_{k}\,\|\,\rho_{k-1}).
\end{align*}
\end{lemmaA}
Applying Lemma~\ref{lem: one_step_improve} with $\rho=\rho^\ast$ yields
\begin{align}\label{eqn: diff_func_val}
0\leq \m F(\rho_{k}) - \m F(\rho^\ast) \leq \frac{1}{\tau_k}\KL(\rho^\ast\,\|\,\rho_{k-1}) - \Big(\frac{1}{\tau_k} + \frac{\lambda}{2}\Big) \KL(\rho^\ast\,\|\,\rho_{k}).
\end{align}
When $\lambda > 0$ and $\tau_k = \tau$ for all $k\geq 1$, the above inequality implies
\begin{align*}
\KL(\rho^\ast\,\|\,\rho_{k}) \leq \Big(1 + \frac{\lambda\tau}{2}\Big)^{-1}\KL(\rho^\ast\,\|\,\rho_{k-1}).
\end{align*}
Therefore, we have
\begin{align*}
\KL(\rho^\ast\,\|\,\rho_k) \leq \Big(1 + \frac{\lambda\tau}{2}\Big)^{-k}\KL(\rho^\ast\,\|\,\rho_0).
\end{align*}
When $\lambda = 0$,~\eqref{eqn: diff_func_val} implies
\begin{align*}
\tau_k\big[\m F(\rho_{k}) - \m F(\rho^\ast)\big] \leq \KL(\rho^\ast\,\|\,\rho_{k-1}) - \KL(\rho^\ast\,\|\,\rho_{k}).
\end{align*}
Summing the above inequality from $1$ to $k$, we have
\begin{align*}
\sum_{l=1}^k \tau_l\big[\m F(\rho_l) - \m F(\rho^\ast)\big] \leq \KL(\rho^\ast\,\|\,\rho_0) - \KL(\rho^\ast\,\|\,\rho_k).
\end{align*}
Therefore, we have
\begin{align*}
\min_{1\leq l\leq k}\m F(\rho_l) - \m F(\rho^\ast) \leq \frac{1}{\tau_1 + \cdots + \tau_k}\KL(\rho^\ast\,\|\,\rho_0).
\end{align*}
\subsection{Proof of Theorem~\ref{thm: discrete}}\label{app: discrete_proof}
Similar to the proof of Theorem~\ref{thm: discrete_regular}, for any $\rho\in\ms P^r(\Theta)$ we have
\begin{align}
\min_{1\leq l\leq k}\m F(\rho_l) - \m F(\rho) \leq \frac{1}{\tau_1 + \cdots + \tau_k}\KL(\rho\,\|\,\rho_0).
\end{align}
By Assumption~\ref{assump: smooth_first_variation}, we have
\begin{align*}
|\m F(\rho^\ast) - \m F(\rho)| \leq LW_1(\rho^\ast, \rho).
\end{align*}
Therefore, 
\begin{align*}
\min_{1\leq l\leq k}\m F(\rho_l) - \m F(\rho^\ast) \leq \frac{1}{\tau_1 + \cdots + \tau_k}\KL(\rho\,\|\,\rho_0) + LW_1(\rho^\ast, \rho).
\end{align*}

\medskip
\noindent {\bf Special cases: discrete measures and singular measures supported on hyperplanes.}
Let $\psi_{m_1}$ be a probability measure on $\Theta$ with first-order moment $m_1 = \mb E_{\theta\sim\psi_{m_1}}\|\theta\| < \infty$. Let $X\sim\rho$ and $Y\sim\psi_{m_1}$. Then $X+Y\sim\rho\ast\psi_{m_1}$. By definition, we have
\begin{align}\label{eqn: W1dist bound}
W_1(\rho, \rho\ast\psi_{m_1}) = \inf_{\theta\sim\rho, \theta'\sim\rho\ast\psi_{m_1}}\mb E\|\theta-\theta'\| \leq \mb E\|X - (X+Y)\| = \mb E\|Y\| = m_1.
\end{align}
This result helps control the smoothing error through $W_1$-distance.

\paragraph{Case 1: $\rho^\ast$ is a discrete measure with bounded support.}
We need the following lemma to control the Gaussian smoothing error in KL divergence. The proof is deferred to Section~\ref{app: lemmas} in this supplementary file.
\begin{lemmaA}[KL divergence bound after Gaussian smoothing]\label{lem: KL smoothing}
Assume $\rho^\ast$ is a discrete probability measure with bounded support. Let $R_\theta = \sup\big\{\|\theta\|: \theta\in\supp(\rho^\ast)\big\}$, and $\rho^\sigma = \rho^\ast\ast\m N(0, \sigma^2I_d)$. If $\rho_0 = \m N(0, \beta^2 I_d)$, we have
\begin{align*}
\KL(\rho^\sigma\,\|\,\rho_0) \leq d\log\frac{\beta}{\sigma} + \frac{d\sigma^2+R_\theta^2}{2\beta^2} - \frac{d}{2}.
\end{align*}
\end{lemmaA}

Note that the first-order moment of $\m N(0, \sigma^2I_d)$ is smaller than $\sqrt{d\sigma^2}$. Applying Lemma~\ref{lem: KL smoothing} and Inequality~\eqref{eqn: W1dist bound} yields
\begin{align*}
0\leq \min_{1\leq l\leq k}\m F(\rho_l) - \m F(\rho^\ast) \leq \frac{1}{\tau_1 + \cdots + \tau_k}\Big(d\log\frac{\beta}{\sigma} + \frac{d\sigma^2+R_\theta^2}{2\beta^2} - \frac{d}{2}\Big) + L\sqrt{d\sigma^2}.
\end{align*}
Since the above inequality holds for all $\sigma > 0$, by choosing 
$\sigma^{2} = L^{-2}(\tau_1 + \cdots + \tau_k)^{-2}$, we have
\begin{align*}
0\leq \min_{1\leq l\leq k}\m F(\rho_l) - \m F(\rho^\ast) 
&\leq \frac{d\log[\beta L(\tau_1 + \cdots + \tau_k)] + \frac{d}{2\beta^2L_2(\tau_1 + \cdots + \tau_k)^2}+\frac{R_\theta^2}{2\beta^2} + \sqrt{d} - \frac{d}{2}}{\tau_1 + \cdots + \tau_k}.
\end{align*}
When $\tau_1 = \cdots = \tau_k = \tau$, the upper bound has order $O(\frac{d\log k}{k})$.

\paragraph{Case 2: $\rho^\ast$ is absolutely continuous with respect to the Lebesgue measure supported on a $d'$-dimensional hyperplane.}
Without loss of generality, assume $\rho^\ast$ is supported on $\supp(\rho^\ast) = \{({\theta'}, 0, \cdots, 0)\in\mb R^d: {\theta'}\in\mb R^{d'}\}$. Let $\rho^\ast_{d'}$ denote the distribution of $\rho^\ast$ restricted to the first $d'$ coordinates. Then $\rho^\ast_{d'}\in\ms P^r(\mb R^{d'})$. Assume $Z = (X, Y)$ with $X\in\mb R^{d'}$ and $Y\in\mb R^{d-d'}$ such that $(X, 0_{d-d'})\sim\rho^\ast$, $Y\sim\m N(0, \sigma^2I_{d-d'})$, and $X$ is independent with $Y$. Then $X\sim\rho^\ast_{d'}$ is a continuous random variable in $\mb R^{d'}$. Similarly, let $Z_0 = (X_0, Y_0) \sim \rho_0 = \m N(0, \beta^2I_d)$, such that $X_0\sim\m N(0, \beta^2I_{d'})$ and $Y_0\sim\m N(0, \beta^2I_{d-d'})$. Then, we have
\begin{align*}
P_Z(z) = P_{X}(x)P_{Y}(y) \quad\mx{and}\quad \rho_0(z) = P_{Z_0}(z) = P_{X_0}(x)P_{Y_0}(y).
\end{align*}
Note that
\begin{align*}
\KL(P_Z\,\|\,\rho_0) &= \KL(P_Z\,\|\,P_{Z_0}) =  \int \log\frac{P_Z}{P_{Z_0}}\,\dd P_Z\\
&= \int\!\!\int \log\frac{P_X(x)P_Y(y)}{P_{X_0}(x)P_{Y_0}(y)}\,\dd P_{X}(x)\,\dd P_{Y}(y)\\
&= \KL(P_X\,\|\,\m N(0, \beta^2I_{d'})) + \KL\big(\m N(0, \sigma^2I_{d-d'})\,\big\|\,\m N(0, \beta^2I_{d-d'})\big)\\
&= \KL(P_X\,\|\,\m N(0, \beta^2I_{d'})) + \frac{d-d'}{2}\Big(\log\frac{\beta^2}{\sigma^2} - 1 + \frac{\sigma^2}{\beta^2}\Big).
\end{align*}
By Theorem~\ref{thm: discrete} and Inequality~\eqref{eqn: W1dist bound}, for every $\sigma^2>0$ we have
\begin{align*}
\min_{1\leq l\leq k}\m F(\rho_l) - \m F(\rho^\ast) \leq \frac{1}{\tau_1 + \cdots + \tau_k}\Big[\KL(P_X\,\|\,\m N(0, \beta^2I_{d'})) + \frac{d-d'}{2}\Big(\log\frac{\beta^2}{\sigma^2} - 1 + \frac{\sigma^2}{\beta^2}\Big)\Big] + L\sqrt{(d-d')\sigma^2}.
\end{align*}
Noting that $P_X = \rho^\ast_{d'}$, by choosing $\sigma^{2} = L^{-2}(\tau_1 + \cdots + \tau_k)^{-2}$, the above inequality implies
\begin{align*}
\min_{1\leq l\leq k}\m F(\rho_l) - \m F(\rho^\ast) \leq \frac{(d-d')\log[\beta L(\tau_1 + \cdots + \tau_k)] + \frac{d-d'}{2\beta^2L_2(\tau_1 + \cdots + \tau_k)^2} + \sqrt{d-d'} - \frac{d-d'}{2} + \KL(\rho^\ast_{d'}\,\|\,\m N(0, \beta^2I_{d'}))}{\tau_1 + \cdots + \tau_k}.
\end{align*}
When $\tau_1 = \cdots = \tau_k = \tau$, the upper bound has order $O\big(\frac{(d-d')\log k}{k}\big)$.
\subsection{Proof of Theorem~\ref{thm: error tol}}
Recall that
\begin{align*}
\eta_k(\theta) = \frac{\delta\m F}{\delta\rho}(\rho_k^{\rm err})(\theta) + \frac{1}{\tau_k}\log\frac{\rho_k^{\rm err}}{\rho_{k-1}^{\rm err}}(\theta).
\end{align*}
Let $\tilde\eta_k = \eta_k - \inf_{\theta\in\Theta}\eta_k(\theta)$. Therefore, we have $\|\tilde\eta_k\|_\infty \leq \varepsilon_k$. Since $\m F$ is $\lambda$-relative strongly convex, we have
\begin{align*}
\m F(\rho) - \m F(\rho_{k}^{\rm err})
&\geq \int_\Theta\frac{\delta\m F}{\delta\rho}(\rho_{k}^{\rm err})(\theta)\,\dd(\rho-\rho_{k}^{\rm err})(\theta) + \frac{\lambda}{2}\KL(\rho\,\|\,\rho_{k}^{\rm err})\\
&= \int_\Theta \tilde\eta_k(\theta) - \frac{1}{\tau_k}\log\frac{\rho_{k}^{\rm err}}{\rho_{k-1}^{\rm err}}(\theta)\,\dd(\rho-\rho_{k}^{\rm err})(\theta) + \frac{\lambda}{2}\KL(\rho\,\|\,\rho_{k}^{\rm err})\\
&= -\frac{1}{\tau_k}\KL(\rho\,\|\,\rho_{k-1}^{\rm err}) + \Big(\frac{1}{\tau_k} + \frac{\lambda}{2}\Big) \KL(\rho\,\|\,\rho_{k}^{\rm err})  + \frac{1}{\tau_k} \KL(\rho_{k}^{\rm err}\,\|\,\rho_{k-1}^{\rm err}) + \int_\Theta \tilde\eta_k(\theta)\,\dd(\rho - \rho_k^{\rm err})(\theta).
\end{align*}
Note that
\begin{align*}
\int_\Theta\tilde\eta_k(\theta)\,\dd(\rho - \rho_k^{\rm err})(\theta)
\leq \|\tilde\eta_k\|_\infty\|\rho - \rho_k^{\rm err}\|_1 \leq \varepsilon_k \cdot\sqrt{2\KL(\rho\,\|\, \rho_{k}^{\rm err})},
\end{align*}
where the last inequality is due to Pinsker's inequality. Thus, we have
\begin{align*}
0\geq \m F(\rho^\ast) - \m F(\rho_k^{\rm err}) \geq -\frac{1}{\tau_k}\KL(\rho^\ast\,\|\,\rho_{k-1}^{\rm err}) + \Big(\frac{1}{\tau_k} + \frac{\lambda}{2}\Big) \KL(\rho^\ast\,\|\,\rho_{k}^{\rm err}) - \varepsilon_k\sqrt{2\KL(\rho^\ast\,\|\, \rho_k^{\rm err})}.
\end{align*}
This implies
\begin{align*}
\sqrt{\KL(\rho^\ast\,\|\,\rho_k^{\rm err})} \leq \frac{\sqrt{\KL(\rho^\ast\,\|\,\rho_{k-1}^{\rm err})}}{\sqrt{1 + \tau_k\lambda/2}} + \sqrt{2}\tau_k\varepsilon_k.
\end{align*}
Therefore,
\begin{align}\label{eqn: errtol_converg}
\sqrt{\KL(\rho^\ast\,\|\,\rho_k^{\rm err})} \leq \frac{\sqrt{\KL(\rho^\ast\,\|\,\rho_0)}}{\prod_{l=1}^k \sqrt{1 + \lambda\tau_l/2}} + \sum_{l=1}^k \frac{\sqrt{2}\tau_l\varepsilon_l}{\prod_{s=l+1}^k\sqrt{1+\lambda\tau_s/2}}.
\end{align}
\paragraph{Case 1:}When $\varepsilon_k \leq \kappa\varepsilon^k$ for some $0 < \varepsilon < 1, \kappa > 0$ and $\tau_k = \tau$,
\begin{align*}
\sum_{l=1}^k \frac{\sqrt{2}\tau_l\varepsilon_l}{\prod_{s=l+1}^k\sqrt{1+\lambda\tau_s/2}}
&\leq \sum_{l=1}^k \frac{\sqrt{2}\tau\kappa\varepsilon^l(1+\lambda\tau/2)^{l/2}}{(1+\lambda\tau/2)^{k/2}} = \frac{\sqrt{2}\tau\kappa}{(1+\lambda\tau/2)^{k/2}}\cdot\sum_{l=1}^k \big[\varepsilon\sqrt{1+\lambda\tau/2}\big]^l.
\end{align*}
We can always assume that $\varepsilon\sqrt{1+\lambda\tau/2}\neq 1$, since if $\varepsilon\sqrt{1+\lambda\tau/2}= 1$, we can find $\varepsilon' \in (\varepsilon, 1)$, so that $\varepsilon_k \leq \kappa(\varepsilon')^k$.
Note that
\begin{align*}
\sum_{l=1}^k  \big[\varepsilon\sqrt{1+\lambda\tau/2}\big]^l \leq 
\begin{cases}
\frac{1}{1-\varepsilon\sqrt{1+\lambda\tau/2}}, & \varepsilon\sqrt{1+\lambda\tau/2} < 1\\
&\\
\frac{[\varepsilon\sqrt{1+\lambda\tau/2}]^{k+1}}{\varepsilon\sqrt{1+\lambda\tau/2} - 1}, & \varepsilon\sqrt{1+\lambda\tau/2} > 1
\end{cases}.
\end{align*}
Therefore, we have
\begin{align*}
\sum_{l=1}^k \frac{\sqrt{2}\tau_l\varepsilon_l}{\prod_{s=l+1}^k\sqrt{1+\lambda\tau_s/2}} \leq
\begin{cases}
\frac{\sqrt{2}\tau\kappa}{1-\varepsilon\sqrt{1+\lambda\tau/2}}\cdot\big(1+\frac{\lambda\tau}{2}\big)^{-\frac{k}{2}}, & \varepsilon\sqrt{1+\lambda\tau/2} < 1\\
&\\
\frac{\sqrt{2}\tau\kappa\varepsilon}{\varepsilon\sqrt{1+\lambda\tau/2}-1}\varepsilon^k, & \varepsilon\sqrt{1+\lambda\tau/2} > 1
\end{cases}
\end{align*}
Therefore, there exists $C = C(\tau, \lambda, \varepsilon) > 0$, such that
\begin{align*}
\sum_{l=1}^k \frac{\sqrt{2}\tau_l\varepsilon_l}{\prod_{s=l+1}^k\sqrt{1+\lambda\tau_s/2}} 
\leq C\kappa \max\{\varepsilon, (1+\lambda\tau/2)^{-1/2}\}^{k}.
\end{align*}
Combining the above inequality with~\eqref{eqn: errtol_converg} yields
\begin{align*}
\KL(\rho^\ast\,\|\,\rho_k^{\rm err}) \leq \frac{C\kappa^2 + 2\KL(\rho^\ast\,\|\,\rho_0)}{(\min\{\varepsilon^{-2}, 1+\lambda\tau/2\})^k}.
\end{align*}
\paragraph{Case 2:}When $\varepsilon_k = \varepsilon k^{-\alpha}$ for some $\alpha, \varepsilon > 0$ and $\tau_k = \tau$ for every $k\geq 1$, we show that $\KL(\rho^\ast\,\|\,\rho_k) \lesssim k^{-2\alpha}$. In fact, note that
\begin{align*}
\sum_{l=1}^k \frac{\sqrt{2}\tau_l\varepsilon_l}{\prod_{s=l+1}^k\sqrt{1+\lambda\tau_s/2}} \leq \frac{\sqrt{2}\tau\varepsilon}{(1+\tau\lambda/2)^{k/2}} \sum_{l=1}^k \frac{(1+\tau\lambda/2)^{l/2}}{l^{\alpha}}.
\end{align*}
We prove that there exists $C = C(\tau,\lambda, \alpha) > 0$, such that 
\begin{align}\label{eqn: bound_of_sum}
    \sum_{l=1}^k \frac{(1+\tau\lambda/2)^{l/2}}{l^\alpha} \leq \frac{C(1+\tau\lambda/2)^{k/2}}{k^\alpha}.
\end{align}
We use the induction to prove~\eqref{eqn: bound_of_sum}. If the statement is correct for $k$, then
\begin{align*}
\sum_{l=1}^{k+1} \frac{(1+\tau\lambda/2)^{l/2}}{l^\alpha} \stackrel{\ri}{\leq} \frac{C(1+\tau\lambda/2)^{k/2}}{k^\alpha} + \frac{(1+\tau\lambda/2)^{(k+1)/2}}{(k+1)^\alpha} \stackrel{\rii}{\leq} \frac{C(1+\tau\lambda/2)^{(k+1)/2}}{(k+1)}.
\end{align*}
Here, (i) is by the induction hypothesis, and (ii) is equivalent to
\begin{align*}
\Big(1+ \frac{1}{k}\Big)^{\alpha} C + \sqrt{1+\frac{\tau\lambda}{2}} \leq C\sqrt{1+\frac{\tau\lambda}{2}}.
\end{align*}
The above inequality is true when
\begin{align*}
\Big(1 + \frac{1}{k}\Big)^\alpha \leq \sqrt{1 + \frac{\tau\lambda}{4}},\quad\mbox{and}\quad C \geq \frac{\sqrt{1+\tau\lambda/2}}{\sqrt{1+\tau\lambda/2} - \sqrt{1+\tau\lambda/4}}.
\end{align*}
When $(1+k^{-1})^{\alpha} >  \sqrt{1+\tau\lambda/4}$, i.e. $k<\big[(1+\tau\lambda/4)^{1/2\alpha}-1\big]^{-1}$, we can choose $C$ large enough such that~\eqref{eqn: bound_of_sum} holds. Therefore, by induction, we know~\eqref{eqn: bound_of_sum} is true for all $k\geq 1$ when
\begin{align*}
C = \max\bigg\{\frac{\sqrt{1+\tau\lambda/2}}{\sqrt{1+\tau\lambda/2} - \sqrt{1+\tau\lambda/4}}, 
\max\Big\{\frac{k^\alpha}{(1+\tau\lambda/2)^{k/2}}\cdot\sum_{l=1}^k \frac{(1+\tau\lambda/2)^{l/2}}{l^\alpha}: k < \frac{1}{(1+\tau\lambda/4)^{1/2\alpha}-1}\Big\}\bigg\}.
\end{align*}
Applying~\eqref{eqn: bound_of_sum} to~\eqref{eqn: errtol_converg} yields
\begin{align*}
\sqrt{\KL(\rho^\ast\,\|\,\rho_k^{\rm err})} 
&\leq \frac{\sqrt{\KL(\rho^\ast\,\|\,\rho_0)}}{(1 + \lambda\tau/2)^{k/2}} + \frac{\sqrt{2}\tau\varepsilon}{(1+\tau\lambda/2)^{k/2}}\cdot \frac{C(1+\tau\lambda/2)^{k/2}}{k^\alpha}\\
&= \frac{\sqrt{\KL(\rho^\ast\,\|\,\rho_0)}}{(1 + \lambda\tau/2)^{k/2}} + \frac{\sqrt{2}C\tau\varepsilon}{k^\alpha}.
\end{align*}
Therefore, we have
\begin{align*}
\KL(\rho^\ast\,\|\,\rho_k^{\rm err}) \leq \frac{2\KL(\rho^\ast\,\|\,\rho_0)}{(1+\lambda\tau/2)^{k}} + \frac{C\varepsilon^2}{k^{2\alpha}}
\end{align*}
for some $C = C(\tau, \lambda, \alpha)$.

\subsection{Proof of Theorem~\ref{thm: stochastic}}
By applying Lemma~\ref{lem: one_step_improve}, we have
\begin{align*}
\Big(\frac{1}{\tau_k} + \frac{\lambda}{2}\Big)\KL(\rho\,\|\,\rho_k^{\rm stoc}) - \frac{1}{\tau_k}\KL(\rho\,\|\,\rho_{k-1}^{\rm stoc}) \leq \m F_{\xi_k}(\rho) - \m F_{\xi_k}(\rho_k^{\rm stoc}) - \frac{1}{\tau_k}\KL(\rho_k^{\rm stoc}\,\|\,\rho_{k-1}^{\rm stoc}).
\end{align*}
Note that
\begin{align*}
\mb E\big[\m F_{\xi_k}(\rho) - \m F_{\xi_k}(\rho_k^{\rm stoc})\,\big|\,\rho_{k-1}^{\rm stoc}\big]
&= \mb E\big[\m F_{\xi_k}(\rho) - \m F_{\xi_k}(\rho_{k-1}^{\rm stoc})\,\big|\,\rho_{k-1}^{\rm stoc}\big] + \mb E\big[\m F_{\xi_k}(\rho_{k-1}^{\rm stoc}) - \m F_{\xi_k}(\rho_k^{\rm stoc})\,\big|\,\rho_{k-1}^{\rm stoc}\big]\\
&\stackrel{\ri}{=} \m F(\rho) - \m F(\rho_{k-1}^{\rm stoc}) + \mb E\big[\m F_{\xi_k}(\rho_{k-1}^{\rm stoc}) - \m F_{\xi_k}(\rho_k^{\rm stoc})\,\big|\,\rho_{k-1}^{\rm stoc}\big]\\
&\stackrel{\rii}{\leq} \m F(\rho) - \m F(\rho_{k-1}^{\rm stoc}) + \mb E\big[L(\xi_k)\sqrt{\KL(\rho_k^{\rm stoc}\,\|\, \rho_{k-1}^{\rm stoc})}\,\big|\,\rho_{k-1}^{\rm stoc}\big]\\
&\stackrel{\riii}{\leq} \m F(\rho) - \m F(\rho_{k-1}^{\rm stoc}) + \sqrt{\mb EL(\xi_k)^2} \cdot \sqrt{\mb E\big[\KL(\rho_k^{\rm stoc}\,\|\, \rho_{k-1}^{\rm stoc})\,\big|\,\rho_{k-1}^{\rm stoc}\big]}.
\end{align*}
Here, both (i) and (ii) are by Assumption~\ref{assump: stoch IKLPD}, and (iii) is by Cauchy--Schwarz inequality. Therefore, we have
\begin{align}
\Big(\frac{1}{\tau_k} + \frac{\lambda}{2}\Big)\mb E\KL(\rho\,\|\,\rho_k^{\rm stoc}) &- \frac{1}{\tau_k}\mb E\KL(\rho\,\|\,\rho_{k-1}^{\rm stoc})\nonumber\\
&\leq \m F(\rho) - \mb E\m F(\rho_{k-1}^{\rm stoc}) + \sqrt{\mb EL(\xi_k)^2}\cdot\sqrt{\mb E\KL(\rho_k^{\rm stoc}\,\|\, \rho_{k-1}^{\rm stoc})} - \frac{1}{\tau_k}\KL(\rho_k^{\rm stoc}\,\|\,\rho_{k-1}^{\rm stoc})\nonumber\\
&\leq \m F(\rho) - \mb E\m F(\rho_{k-1}^{\rm stoc}) + \frac{\tau_k}{4}\mb EL(\xi)^2.\label{eqn: stoch_onestep}
\end{align}
When $\lambda = 0$,~\eqref{eqn: stoch_onestep} implies
\begin{align*}
\tau_k\big[\mb E\m F(\rho_{k-1}^{\rm stoc}) - \m F(\rho^\ast)\big] \leq \mb E\KL(\rho^\ast\,\|\,\rho_{k-1}^{\rm stoc}) - \mb E\KL(\rho^\ast\,\|\,\rho_k^{\rm stoc}) + \frac{\tau_k^2}{4}\mb EL(\xi)^2.
\end{align*}
Therefore, we have
\begin{align*}
\min_{0\leq l\leq k-1} \mb E\m F(\rho_l^{\rm stoc}) - \m F(\rho^\ast) \leq \frac{\KL(\rho^\ast\,\|\,\rho_0)}{\tau_1 + \cdots + \tau_k} + \frac{\tau_1^2 + \cdots + \tau_k^2}{4(\tau_1 + \cdots + \tau_k)}\mb EL(\xi)^2.
\end{align*}
By taking $\tau_k = \frac{\tau}{\sqrt{k}}$ and using $k^{-1/2}\geq 2\sqrt{k+1} - 2\sqrt{k}$, the above inequality implies that
\begin{align*}
\min_{0\leq l\leq k-1} \mb E\m F(\rho_l^{\rm stoc}) - \m F(\rho^\ast) \leq \frac{4\KL(\rho^\ast\,\|\,\rho_0) + \tau^2\log(k+1)\mb EL(\xi)^2}{8\tau(\sqrt{k+1}-1)}.
\end{align*}
When $\lambda > 0$,~\eqref{eqn: stoch_onestep} implies
\begin{align*}
\mb E\m F(\rho_{k-1}^{\rm stoc}) - \m F(\rho^\ast) \leq \frac{1}{\tau_k}\mb E\KL(\rho^\ast\,\|\,\rho_{k-1}^{\rm stoc}) - \Big(\frac{1}{\tau_k} + \frac{\lambda}{2}\Big)\mb E\KL(\rho^\ast\,\|\,\rho_k^{\rm stoc}) + \frac{\tau_k}{4}\mb EL(\xi)^2.
\end{align*}
By taking $\tau_k = \frac{2}{\lambda(k+1)}$, we have
\begin{align*}
\min_{0\leq l\leq k-1}\mb E\m F(\rho_l^{\rm stoc}) - \m F(\rho^\ast) &\leq \frac{\lambda}{k}\KL(\rho^\ast\,\|\,\rho_0) + \frac{\mb EL(\xi)^2}{4k}\sum_{l=1}^k \frac{2}{\lambda(l+1)}\\
&\leq \frac{\lambda}{k}\KL(\rho^\ast\,\|\,\rho_0) + \frac{\log(k+1)}{2\lambda k}\mb EL(\xi)^2.
\end{align*}
In the last inequality, we use $\frac{1}{l+1} \leq \log\frac{l+1}{l}$ for all $l \geq 1$.
\subsection{Proofs of technical results}\label{app: lemmas}
\begin{proof}[Proof of Lemma~\ref{lem: one_step_improve}]
By first-order optimality condition of~\eqref{eqn: IKLPD}, we know that
\begin{align*}
\frac{\delta\m F}{\delta\rho}(\rho_{k}) + \frac{1}{\tau_k}\log\frac{\rho_{k}}{\rho_{k-1}}
\end{align*}
is a constant. Since $\m F$ is $\lambda$-relative strongly convex, we have
\begin{align*}
\m F(\rho) - \m F(\rho_{k})
&\geq \int_\Theta\frac{\delta\m F}{\delta\rho}(\rho_{k})(\theta)\,\dd(\rho-\rho_{k})(\theta) + \frac{\lambda}{2}\KL(\rho\,\|\,\rho_{k})\\
&= -\frac{1}{\tau_k}\int_\Theta\log\frac{\rho_{k}}{\rho_{k-1}}(\theta)\,\dd(\rho-\rho_{k})(\theta) + \frac{\lambda}{2}\KL(\rho\,\|\,\rho_{k})\\
&= -\frac{1}{\tau_k}\KL(\rho\,\|\,\rho_{k-1}) + \Big(\frac{1}{\tau_k} + \frac{\lambda}{2}\Big) \KL(\rho\,\|\,\rho_{k})  + \frac{1}{\tau_k} \KL(\rho_{k}\,\|\,\rho_{k-1}).
\end{align*}
\end{proof}

\begin{proof}[Proof of Lemma~\ref{lem: KL smoothing}]
Since $\rho^\ast$ is discrete probability measure, $\rho^\sigma = \rho^\ast\ast\m N(0, \sigma^2I_d)$ is a Gaussian mixture distribution. The main step is to prove
\begin{align}\label{eqn: merge_mixture}
\KL(\rho^\sigma\,\|\,\rho_0) \leq \sup_{\theta\in\supp(\rho^\ast)}\KL\Big(\m N\big(\theta, \sigma^2I_d\big)\,\Big\|\,\rho_0\Big).
\end{align}
In fact, for any $\theta_j, \theta_l\in\supp(\rho^\ast)$ with $\theta_j\neq\theta_l$, assume $w_j = \rho^\ast(\theta_j)$ and $w_l = \rho^\ast(\theta_l)$. Let
\begin{align*}
\rho^\ast_{-jl} = \sum_{\substack{\theta\in\supp(\rho^\ast)\\ \theta\neq\theta_j, \theta_l}}\rho^\ast(\theta)\delta_{\theta}
\end{align*}
be a measure by deleting the contribution of $\theta_j$ and $\theta_l$ in $\rho^\ast$. (Note that $\rho^\ast_{-jl}(\Theta) = 1-w_j-w_l < 1$, so $\rho^\ast_{-jl}$ is not a probability measure.)
Consider the optimization problem
\begin{align*}
\max_{\substack{w + w' = w_j + w_l\\ w, w'\geq 0}}g_{jl}(w, w') 
&\coloneqq\KL\Big(\big(\rho^\ast_{-jl} + w\delta_{\theta_j} + w'\delta_{\theta_l}\big)\ast\m N(0,\sigma^2)\,\Big\|\,\rho_0\Big)\\
&= \KL\Big(\rho^\ast_{-jl}\ast\m N(0, \sigma^2) + w\m N(\theta_j, \sigma^2I_d) + w'\m N(\theta_l, \sigma^2I_d)\,\Big\|\,\rho_0\Big).
\end{align*}
It is easy to see that $g_{jl}$ is a convex function on $\{(w, w')\subset\mb R_{\geq 0}^2: w+w' = w_j+w_l\}$. Therefore, $g_{jl}$ achieves its maximum on the boundary $(w, w') = (w_j+w_l, 0)$ or $(w, w') = (0, w_j+w_l)$. The above argument indicates that we can always merge two mixtures of $\rho^\sigma$ into one while the KL divergence is not decreasing. Therefore, the inequality~\eqref{eqn: merge_mixture} holds.
Applying~\eqref{eqn: merge_mixture}, we know
\begin{align*}
\KL(\rho^\sigma\,\|\,\rho_0) 
&\leq \sup_{\theta\in\supp(\rho^\ast)}\KL\Big(\m N\big(\theta, \sigma^2I_d\big)\,\Big\|\,\rho_0\Big)\\
&=\sup_{\theta\in\supp(\rho^\ast)} \frac{1}{2}\Big(\log\frac{\det(\beta^2I_d)}{\det(\sigma^2I_d)} - d + \tr(\beta^{-2}\sigma^2I_d) + \theta^\top(\beta^2I_d)^{-1}\theta\Big)\\
&= \frac{1}{2}\Big(2d\log\frac{\beta}{\sigma} - d + \frac{d\sigma^2}{\beta^2} + \frac{\|\theta\|^2}{\beta^2}\Big)\\
&\leq d\log\frac{\beta}{\sigma} + \frac{d\sigma^2 + R_\theta^2}{2\beta^2} - \frac{d}{2}.
\end{align*}
\end{proof}

\subsection{Convexity of NPMLE and KL Divergence}\label{app:convexity}
\paragraph{NPMLE.} Recall that the empirical loss function in NPMLE is
\begin{align*}
\m L_n(\rho) = -\frac{1}{n}\sum_{i=1}^n\log\Big(\int_\Theta p(X_i\,|\,\theta)\,\dd \rho(\theta)\Big).
\end{align*}
Then, for any $\rho, \rho'\in\ms P(\Theta)$ and $t\in[0, 1]$, we have
\begin{align*}
\m L_n\big((1-t)\rho + t\rho'\big)
&= -\frac{1}{n}\sum_{i=1}^n \log\Big((1-t)\int_\Theta p(X_i\,|\,\theta)\,\dd\rho(\theta) + t\int_\Theta p(X_i\,|\,\theta)\,\dd\rho'(\theta)\Big)\\
&\stackrel{\ri}{\leq} -\frac{1-t}{n}\sum_{i=1}^n\log\Big(\int_\Theta p(X_i\,|\,\theta)\,\dd \rho(\theta)\Big) - \frac{t}{n}\sum_{i=1}^n\log\Big(\int_\Theta p(X_i\,|\,\theta)\,\dd \rho'(\theta)\Big)\\
&= (1-t)\m L_n(\rho) + t\m L_n(\rho').
\end{align*}
Here, (i) is due to the convexity of function $x\mapsto -\log x$. The above inequality implies
\begin{align*}
\m L_n(\rho') - \m L_n(\rho) \geq \frac{\m L_n\big(\rho + t(\rho'-\rho)\big) - \m L_n(\rho)}{t}, \quad\forall\,t\in[0, 1].
\end{align*}
By the definition of first-order variation and letting $t\to 0^+$ on the right-hand side yield
\begin{align*}
\m L_n(\rho') - \m L_n(\rho) \geq \lim_{t\to0^+}\frac{\m L_n\big(\rho + t(\rho'-\rho)\big) - \m L_n(\rho)}{t} = \int\frac{\delta\m F}{\delta\rho}(\rho)\,\dd(\rho'-\rho).
\end{align*}
Therefore, $\m L_n$ is ($L_2$-)convex.

\paragraph{KL divergence} For any $\pi\in\ms P^r(\Theta)$, we will show that $\KL(\cdot\,\|\,\pi)$ is $1$-relative strongly convex. We provide the proof to make our paper self-contained. For any $\rho, \rho'\in\ms P^r(\Theta)$, we have
\begin{align*}
\KL(\rho'\,\|\,\pi) - \KL(\rho\,\|\,\pi)
&= \int_\Theta -\log\pi \,\dd(\rho' - \rho) + \KL(\rho'\,\|\,\rho) + \int_\Theta\log\rho\,\dd(\rho'-\rho)\\
&= \int_\Theta \frac{\delta\KL(\cdot\,\|\,\pi)}{\delta\rho}(\rho)\,\dd(\rho'-\rho) + \KL(\rho'\,\|\,\rho).
\end{align*}
In the last equation, we use the fact that 
\begin{align*}
\frac{\delta\KL(\cdot\,\|\,\pi)}{\delta\rho}(\rho) = \log\rho - \log\pi.
\end{align*}
In fact,~\cite{nitanda2022convex, chizat2022mean} show a stronger result that for any convex functional $\m H$, the functional $\m F(\rho) = \m H(\rho) + \lambda\int\rho\log\rho$ is $\lambda$-relative strongly convex. In the KL divergence case, we can simply take $\m H(\rho) = -\int_\Theta\log\pi\,\dd\rho$ and $\lambda = 1$.


\end{document}